\documentclass{amsart}
\usepackage{amsfonts,amscd,amsthm,amsgen,amsmath,amssymb}
\usepackage[all]{xy}
\usepackage{epsfig,color}
\usepackage[vcentermath]{youngtab}

\newtheorem{theorem}{Theorem}[section]
\newtheorem{lemma}[theorem]{Lemma}
\newtheorem{corollary}[theorem]{Corollary}
\newtheorem{proposition}[theorem]{Proposition}

\theoremstyle{definition}
\newtheorem{definition}[theorem]{Definition}

\theoremstyle{remark}
\newtheorem{remark}[theorem]{Remark}

\numberwithin{equation}{section}

\begin{document}

\title{Brieskorn spheres, cyclic group actions and the Milnor conjecture}

\author{David Baraglia}
\address{School of Computer and Mathematical Sciences, The University of Adelaide, Adelaide SA 5005, Australia}
\email{david.baraglia@adelaide.edu.au}

\author{Pedram Hekmati}
\address{Department of Mathematics, The University of Auckland, Auckland, 1010, New Zealand}
\email{p.hekmati@auckland.ac.nz}


\date{\today}

\begin{abstract}
In this paper we further develop the theory of equivariant Seiberg--Witten--Floer cohomology of the two authors, with an emphasis on Brieskorn homology spheres. We obtain a number of applications. First, we show that the knot concordance invariants $\theta^{(c)}$ defined by the first author satisfy $\theta^{(c)}(T_{a,b}) = (a-1)(b-1)/2$ for torus knots, whenever $c$ is a prime not dividing $ab$. Since $\theta^{(c)}$ is a lower bound for the slice genus, this gives a new proof of the Milnor conjecture. Second, we prove that a free cyclic group action on a Brieskorn homology $3$-sphere $Y = \Sigma(a_1 , \dots , a_r)$ does not extend smoothly to any homology $4$-ball bounding $Y$. In the case of a non-free cyclic group action of prime order, we prove that if the rank of $HF_{red}^+(Y)$ is greater than $p$ times the rank of $HF_{red}^+(Y/\mathbb{Z}_p)$, then the $\mathbb{Z}_p$-action on $Y$ does not extend smoothly to any homology $4$-ball bounding $Y$. Third, we prove that for all but finitely many primes a similar non-extension result holds in the case that the bounding $4$-manifold has positive definite intersection form. Finally, we also prove non-extension results for equivariant connected sums of Brieskorn homology spheres.
\end{abstract}

\maketitle


\section{Introduction}

In \cite{bh}, we introduced the theory of equivariant Seiberg--Witten--Floer cohomology and established its basic properties. In this paper we further develop this theory, with a particular emphasis on Brieskorn homology spheres. Applications include a new proof of the Milnor conjecture and obstructions to extending group actions over a bounding $4$-manifold.

For pairwise coprime positive integers $a_1 , \dots , a_r > 1$, the Brieskorn manifold $Y = \Sigma(a_1 , \dots , a_r)$ is an integral homology Seifert $3$-manifold. The Seifert structure defines a circle action on $Y$. Restricting the circle action to finite subgroups, we obtain an action of the cyclic group $\mathbb{Z}_p$ on $Y$ for each integer $p > 1$. We obtain our main results by considering the equivariant Seiberg--Witten--Floer cohomology of $Y$ with respect to such $\mathbb{Z}_p$-actions.

\subsection{Knot concordance invariants and the Milnor conjecture}

Let $p$ be a prime number. One way of producing $\mathbb{Z}_p$-actions on rational homology $3$-spheres is to take $Y = \Sigma_p(K)$, the cyclic $p$-fold cover of $S^3$ branched over a knot $K$. From the equivariant Seiberg--Witten--Floer cohomology of $Y$ one may extract invariants of the knot $K$. In \cite{bar}, this construction was used to obtain a series of knot concordance invariants $\theta^{(p)}(K)$. These invariants are lower bounds for the slice genus $g_4(K)$, that is, $g_4(K) \ge \theta^{(p)}(K)$ for all primes $p$. More generally, the invariants $\theta^{(p)}$ can be used to bound the genus of surfaces bounding $K$ in negative definite $4$-manifolds with $S^3$ boundary.

We are interested in the case that $K$ is a torus knot. If $K = T_{a,b}$ is an $(a,b)$ torus knot and $c$ is a prime not dividing $ab$, then $\Sigma_c(T_{a,b})$ is the Brieskorn homology sphere $\Sigma(a,b,c)$ and the $\mathbb{Z}_c$-action arising from the branched covering construction coincides with the restriction to $\mathbb{Z}_c$ of the Seifert circle action. By studying the equivariant Seiberg--Witten--Floer homology of $\Sigma(a,b,c)$, we deduce the following:

\begin{theorem}\label{thm:theta1}
Let $a,b > 1$ be coprime integers and let $c$ be a prime not dividing $ab$. Then $\theta^{(c)}(T_{a,b}) = \dfrac{1}{2}(a-1)(b-1)$.
\end{theorem}

Since $\theta^{(c)}$ is a lower bound for the slice genus, we obtain the Milnor conjecture as an immediate corollary:
\begin{corollary}
Let $a,b > 1$ be coprime. Then $g_4(T_{a,b}) = \dfrac{1}{2}(a-1)(b-1)$.
\end{corollary}
\begin{proof}
The Milnor fibre of the singularity $x^a = y^b$ has genus $(a-1)(b-1)/2$ \cite{mil}, hence $g_4(T_{a,b}) \le (a-1)(b-1)/2$. On the other hand, if we let $c$ be any prime not dividing $ab$, then Theorem \ref{thm:theta1} gives $g_4(T_{a,b}) \ge (a-1)(b-1)/2$.
\end{proof}

The original proof of the Milnor conjecture due to Kronheimer and Mrowka uses gauge theory and adjunction inequalities \cite{krmr}. The result was proven again by Ozsv\'ath and Szab\'o using the $\tau$-invariant of Knot Floer homology \cite{os1} and by Rasmussen using the $s$-invariant of Khovanov homology \cite{ras}. Although our proof uses gauge theory, it does not use adjunction inequalities but rather is based on Floer theoretic methods. Thus our proof has more in common with Ozsv\'ath--Szab\'o and Rasmussen than with Kronheimer--Mrowka. It is interesting to note that our proof, like those of Ozsv\'ath--Szab\'o and Rasmussen, is based on finding a knot concordance invariant which bounds the slice genus and equals $(a-1)(b-1)/2$ for the torus knot $T_{a,b}$.

\subsection{Equivariant delta invariants of Brieskorn homology spheres}

Our next result concerns the equivariant delta invariants of Brieskorn homology spheres. The equivariant delta invariants, introduced in \cite{bh}, are a certain equivariant generalisation of the Ozsv\'ath--Szab\'o $d$-invariant and are equivariant homology cobordism invariants. Given a rational homology $3$-sphere $Y$, an action of $\mathbb{Z}_p$ on $Y$ by orientation preserving diffeomorphisms and a $\mathbb{Z}_p$-invariant spin$^c$-structure $\mathfrak{s}$, we obtain a sequence of invariants $\delta_j^{(p)}(Y , \mathfrak{s}) \in \mathbb{Q}$ indexed by a non-negative integer $j$. We call $\delta_j^{(p)}(Y , \mathfrak{s})$ the equivariant delta invariants of $(Y , \mathfrak{s})$. When $Y$ is an integral homology $3$-sphere, it has a unique spin$^c$-structure which is automatically $\mathbb{Z}_p$-invariant. In this case we may write the invariants as $\delta^{(p)}_j(Y)$. The most important property of these invariants is that they satisfy an equivariant version of the Fr{\o}yshov inequality \cite{bh}. In particular, this implies that they are invariant under equivariant homology cobordism. Consequently, the $\delta_j^{(p)}$ define obstructions to extending the $\mathbb{Z}_p$-action over an integral or rational homology $4$-ball bounding $Y$:
\begin{proposition}[Proposition 7.6, \cite{bh}]
Let $Y$ be an integral homology $3$-sphere on which $\mathbb{Z}_p$ acts by orientation preserving diffeomorphisms. Suppose that $Y$ is bounded by a smooth integer homology $4$-ball $W$. If the $\mathbb{Z}_p$-action extends smoothly over $W$ then $\delta_j^{(p)}(Y) = \delta_j^{(p)}(-Y) = 0$ for all $j \ge 0$.
\end{proposition}

In fact, we can relax the assumption that $W$ is an integer homology $4$-ball to simply being a rational homology $4$-ball provided that $W$ admits a $\mathbb{Z}_p$-invariant spin$^c$-structure. This is automatically true if $p$ does not divide the order of $H^2(W ; \mathbb{Z})$, for then $\mathbb{Z}_p$ can't act freely on the set of spin$^c$-structures.

The sequence of invariants $\{ \delta_j^{(p)}(Y) \}_{j \ge 0}$ is decreasing and eventually constant. We set $\delta_\infty^{(p)}(Y) = \lim_{j \to \infty} \delta_j^{(p)}(Y)$.

Let $Y = \Sigma(a_1, \dots , a_r)$ be a Brieskorn homology sphere and let $p$ be any prime. We assume $a_1, \dots , a_r > 1$ and $r \ge 3$ so that $Y \neq S^3$. In \textsection \ref{sec:deltabri} we prove the following results (see Proposition \ref{prop:deltaprop}):

\begin{itemize}
\item[(1)]{$\delta_j^{(p)}(Y) = \delta_\infty^{(p)}(Y)$ for all $j \ge 0$.}
\item[(2)]{$\delta(Y) \le \delta_\infty^{(p)}(Y) \le -\lambda(Y)$.}
\item[(3)]{$\lambda(Y) \le \delta_j^{(p)}(-Y) \le -\delta(Y)$ for all $j \ge 0$.}
\end{itemize}
Here $\lambda(Y)$ is the Casson invariant of $Y$ and $\delta(Y) = d(Y)/2$ is half the Ozsv\'ath--Szab\'o $d$-invariant. In particular, if $Y$ is a Brieskorn homology sphere which bounds a contractible $4$-manifold, then $\delta(Y) = 0$, $\delta_j^{(p)}(Y) = \delta_\infty^{(p)}(Y)$ and
\[
\delta_\infty^{(p)}(-Y) \le \delta_j^{(p)}(-Y) \le 0
\]
for all $j \ge 0$. Thus $\delta_j^{(p)}(\pm Y) = 0$ for all $j$ if and only if $\delta_\infty^{(p)}( \pm Y) = 0$. This justifies restricting attention to the invariants $\delta_\infty^{(p)}( \pm Y)$.

We first consider the case of free $\mathbb{Z}_p$-actions. Given a prime $p$, the restriction of the Seifert circle action on $Y = \Sigma(a_1 , \dots , a_r)$ to $\mathbb{Z}_p$ acts freely if and only if $p$ does not divide $a_1 \cdots a_r$. Hence the action is free for all but finitely many primes. In fact any free action of a finite group on $Y$ is conjugate to a finite subgroup of the Seifert circle action \cite[Proposition 4.3]{lusj}. In the free case we have:

\begin{theorem}\label{thm:free0}
Let $Y = \Sigma(a_1, \dots , a_r)$ be a Brieskorn homology sphere and let $p$ be a prime not dividing $a_1 \cdots a_r$. Set $Y_0 = Y/\mathbb{Z}_p$. Then for any spin$^c$-structure $\mathfrak{s}_0$ on $Y_0$, we have 
\[
\delta^{(p)}_\infty(Y) - \delta(Y) = {\rm rk}( HF_{red}^+(Y) ) - {\rm rk}( HF_{red}^+(Y_0 , \mathfrak{s}_0 ) ).
\]
\end{theorem}

Furthermore, we have:

\begin{theorem}\label{thm:free1}
We have that ${\rm rk}( HF_{red}^+(Y) ) > {\rm rk}( HF_{red}^+( Y_0 , \mathfrak{s}_0 ))$ except in the following cases:
\begin{itemize}
\item[(1)]{$Y = \Sigma(2,3,5)$ and $p$ is any prime.}
\item[(2)]{$Y = \Sigma(2,3,11)$ and $p=5$.}
\end{itemize}
In case (1) we have ${\rm rk}( HF_{red}^+(Y) ) = {\rm rk}( HF_{red}^+( Y_0 , \mathfrak{s}_0 )) = 0$ and in case (2) we have ${\rm rk}( HF_{red}^+(Y) ) = {\rm rk}( HF_{red}^+( Y_0 , \mathfrak{s}_0 )) = 1$.
\end{theorem}

Combining these two results gives:
\begin{corollary}\label{cor:dinffree0}
Let $Y = \Sigma(a_1 , a_2 , \dots , a_r)$ be a Brieskorn homology sphere and let $p$ be a prime not dividing $a_1 \cdots a_r$. Then $\delta_\infty^{(p)}(Y) > \delta(Y)$ except in the following cases:
\begin{itemize}
\item[(1)]{$Y = \Sigma(2,3,5)$ and $p$ is any prime.}
\item[(2)]{$Y = \Sigma(2,3,11)$ and $p=5$.}
\end{itemize}
In both cases we have $\delta_\infty^{(p)}(Y) = \delta(Y) = 1$.
\end{corollary}

\begin{corollary}\label{cor:ne}
Let $Y = \Sigma(a_1 , a_2 , \dots , a_r)$ be a Brieskorn homology sphere and let $m>1$ be an integer not dividing $a_1 \cdots a_r$. Suppose that $W$ is smooth rational homology $4$-ball bounding $Y$ and that $m$ does not divide $| H^2(X ; \mathbb{Z})|$. Then the $\mathbb{Z}_m$-action on $Y$ does not extend smoothly to $W$. 
\end{corollary}
\begin{proof}
Let $p$ be a prime divisor of $m$ which does not divide $| H^2( X ; \mathbb{Z})|$. It suffices to show that the subgroup $\mathbb{Z}_p \subseteq \mathbb{Z}_m$ does not extend over $W$. Since $Y$ is bounded by a rational homology $4$-ball we have $\delta(Y) = 0$. Then Corollary \ref{cor:dinffree0} implies that $\delta_\infty^{(p)}(Y) > 0$, unless $Y = \Sigma(2,3,5)$ or $Y = \Sigma(2,3,11)$ and $p=5$. However these cases do not bound rational homology $4$-balls as they have $\delta(Y) = 1$. So $\delta_\infty^{(p)}(Y) > 0$ which implies that the $\mathbb{Z}_p$-action does not extend smoothly to $W$.
\end{proof}

The $r=3$ case of the above result was proven by Anvari--Hambleton \cite{anha1}, under the assumption that the bounding manifold is contractible. We note that there is a conjecture that Brieskorn spheres with $r > 3$ can not bound contractible manifolds (see, for example \cite[Problem I]{savk}). On the other hand, there are many examples of Brieskorn spheres which bound rational homology balls, but not integer homology balls \cite{fs2,al,savk2}. Thus Corollary \ref{cor:ne} is a non-trivial result.

We can also show that for all sufficiently large primes, $\delta_\infty^{(p)}(Y)$ equals minus the Casson invariant $-\lambda(Y)$.

\begin{theorem}\label{thm:free2}
Let $Y = \Sigma(a_1 , a_2 , \dots , a_r)$ be a Brieskorn homology sphere and let $p$ be a prime not dividing $a_1 \cdots a_r$. Suppose that $p > N$ where 
\[
N = a_1 \cdots a_r \left( (r-2) - \sum_{i=1}^r \frac{1}{a_i} \right).
\]
Then $\delta_\infty^{(p)}(Y) = -\lambda(Y)$.
\end{theorem}
\begin{proof}
This follows from Theorem \ref{thm:free0} and Proposition \ref{prop:deltap}.
\end{proof}

We now consider the case of branched coverings. Let $Y = \Sigma(a_1, \dots , a_r)$ be a Brieskorn homology sphere and let $p$ be a prime dividing $a_1 \cdots a_r$. Without loss of generality we may assume that $p$ divides $a_1$. Then the quotient space $Y_0 = Y/\mathbb{Z}_p$ is the Brieskorn homology sphere $\Sigma(a_1/p , a_2 , \dots , a_r)$ and $Y \to Y_0$ is a cyclic branched covering. Our main result is the following:

\begin{theorem}\label{thm:deltabranch}
We have that 
\[
\delta(-Y) - \delta^{(p)}_\infty(-Y) \ge {\rm rk}( HF_{red}^+(Y)) - p \, {\rm rk}( HF_{red}^+(Y_0) ).
\]
\end{theorem}

\begin{remark}
Karakurt--Lidman have shown that ${\rm rk}( HF_{red}^+(Y)) \ge p \, {\rm rk}( HF_{red}^+(Y_0) )$. Thus the right hand side of the inequality in Theorem \ref{thm:deltabranch} is non-negative. Moreover, if $\delta_\infty^{(p)}(-Y) = \delta(-Y)$ then we must have an equality: ${\rm rk}( HF_{red}^+(Y)) = p \, {\rm rk}( HF_{red}^+(Y_0) )$.
\end{remark}

\begin{theorem}
Let $Y = \Sigma(a_1 , a_2 , \dots , a_r)$ be a Brieskorn homology sphere and let $p$ be a prime dividing $a_1 \cdots a_r$. Suppose that $W$ is a rational homology $4$-ball bounding $Y$ and that $p$ does not divide the order of $H^2(W;\mathbb{Z})$. If ${\rm rk}( HF_{red}^+(Y)) > p \, {\rm rk}( HF_{red}^+(Y_0) )$, then the $\mathbb{Z}_p$-action on $Y$ does not extend smoothly to $W$.
\end{theorem}

The $r=3$ case of this result was proved by Anvari--Hambleton \cite{anha2} (in the integral homology case) without requiring the assumption that ${\rm rk}( HF_{red}^+(Y)) > p \, {\rm rk}( HF_{red}^+(Y_0) )$.

We expect that the condition ${\rm rk}( HF_{red}^+(Y)) = p \, {\rm rk}( HF_{red}^+(Y_0) )$ is rarely satisfied, however there are some cases where it does hold. One family of examples is given by $Y = \Sigma(2,3,30n+5)$ and $p=5$, in which case $Y_0 = \Sigma(2,3,6n+1)$ and ${\rm rk}( HF_{red}^+(Y)) = 5 \, {\rm rk}( HF_{red}^+(Y_0) ) = 5n$. All of these examples have $\delta(Y) = 1$, so such a $Y$ can not bound a contractible $4$-manifold. We suspect that there are no examples where ${\rm rk}( HF_{red}^+(Y)) = p \, {\rm rk}( HF_{red}^+(Y_0) )$ and $Y$ bounds an integral homology $4$-ball.

\subsection{Non-extension results for positive definite $4$-manifolds}

Our equivariant $\delta$-invariants can also be used to obstruct the extension of the $\mathbb{Z}_p$-action over a positive definite $4$-manifold bounding $Y$. First, we have the following result, which is a consequence of \cite[Theorem 5.3]{bh}:

\begin{proposition}\label{prop:wy}
Let $Y$ be an integral homology $3$-sphere on which $\mathbb{Z}_p$ acts by orientation preserving diffeomorphisms. Suppose that $Y$ is bounded by a smooth, compact, oriented, $4$-manifold $W$ with positive definite intersection form and with $b_1(W) = 0$. Suppose that the $\mathbb{Z}_p$ extends to a smooth, homologically trivial action on $W$. Then
\[
\min_c \left\{ \frac{ c^2 - rk( H^2(W ; \mathbb{Z}) ) }{8}  \right\} \ge \delta_0^{(p)}(Y)
\]
where the minimum is taken over all characteristic elements of $H^2( W ; \mathbb{Z})$.
\end{proposition}
\begin{proof}
Suppose $\mathbb{Z}_p$ extends smoothly and homologically trivially to $W$. Take any characteristic $c \in H^2(X ; \mathbb{Z})$. Then there is a unique spin$^c$-structure $\mathfrak{s}$ with $\mathfrak{s} = c$. Since the action is homologically trivial it follows that $\mathfrak{s}$ is $\mathbb{Z}_p$-invariant. Now we apply \cite[Theorem 5.3]{bh} to $-W$ with $Y$ regarded as an ingoing boundary to obtain: $(c^2 - rk(H^2(W ; \mathbb{Z})))/8 \ge \delta_0^{(p)}(Y)$. Taking the minimum over all characteristics gives the result.
\end{proof}

\begin{corollary}\label{cor:ne2}
Let $W$ and $Y$ be as in Proposition \ref{prop:wy}. If $\delta_0^{(p)}(Y) > 0$ or $\delta^{(p)}_\infty(-Y) < 0$. Then the $\mathbb{Z}_p$-action on $Y$ does not extend smoothly and homologically trivially to $W$.
\end{corollary}
\begin{proof}
First note that $\delta_0^{(p)}(Y) + \delta_\infty^{(p)}(-Y) \ge \delta_\infty^{(p)}(Y) + \delta_\infty^{(p)}(-Y) \ge 0$ \cite{bh}. So if $\delta_\infty^{(p)}(-Y) < 0$, then $\delta_0^{(p)}(Y) > 0$. So we can assume that $\delta_0^{(p)}(Y) > 0$. If the $\mathbb{Z}_p$-action extends smoothly and homologically trivially to $W$, then 
\[
\min_c \left\{ \frac{ c^2 - rk( H^2(W ; \mathbb{Z}) ) }{8}  \right\} \ge \delta_0^{(p)}(Y) > 0.
\]
But $Y$ is an integral homology $3$-sphere, so $H^2(W ; \mathbb{Z})$ is a unimodular integral lattice. A result of Elkies \cite{elk} implies that $\min_c  \{  (c^2 - rk( H^2(W ; \mathbb{Z}) ) )/8 \} \le 0$, which is a contradiction.
\end{proof}

Combined with our calculation of $\delta_0(Y), \delta_\infty(-Y)$ for Brieskorn spheres, we obtain the following non-extension result:

\begin{corollary}
Let $Y = \Sigma( a_1 , \dots , a_r)$ be a Brieskorn homology sphere and $p$ any prime. If $\delta_\infty(Y) > 0$ or $\delta_\infty(-Y) > 0$, then the $\mathbb{Z}_p$-action on $Y$ does not extend smoothly and homologically trivial to any smooth, compact, oriented, $4$-manifold $W$ with positive definite intersection form and with $b_1(W) = 0$ bounding $Y$.
\end{corollary}

In particular, if $p$ does not divide $a_1 \cdots a_r$ and $p > N = a_1 \cdots a_r \left( (r-2) - \sum_{i=1}^r \frac{1}{a_i} \right)$, then $\delta_\infty(Y) = -\lambda(Y) > 0$ by Theorem \ref{thm:free2}. Hence for a given $Y$, the above non-extension result applies to all but finitely many primes.

Note that such a non-extension result does not exist if $W$ has a negative definite intersection form. Indeed for any $p$, the $\mathbb{Z}_p$-action on $Y$ extends smoothly and homologically trivially to any negative definite star-shaped plumbing bounding $Y$ \cite[\textsection 2]{or}.

\subsection{Non-extension results for connected sums}

Suppose that $Y_1, \dots , Y_m$ are Brieskorn homology spheres and $p$ is a prime such that for each $i$, the $\mathbb{Z}_p$-action on $Y_i$ is not free. Then we can form an equivariant connected sum $Y = Y_1 \# \cdots \# Y_m$ by attaching the summands to each other along fixed points. From \cite[Proposition 3.1]{bar} we have that $\delta_{j_1+ \cdots + j_m}^{(p)}(-Y) \le \sum_{k=1}^m \delta_{j_k}^{(p)}(-Y_k)$. Taking $j_1, \dots , j_m$ sufficiently large, we obtain
\[
\delta_{\infty}^{(p)}(-Y) \le \sum_{k=1}^m \delta_{\infty}^{(p)}(-Y_k).
\]
Then Theorem \ref{thm:deltabranch} implies that
\begin{equation}\label{equ:dest}
\delta(-Y) - \delta_{\infty}^{(p)}(-Y) \ge \sum_{k=1}^m  \left( {\rm rk}( HF_{red}^+(Y_k)) - p \, {\rm rk}( HF_{red}^+(Y_k/\mathbb{Z}_p) ) \right),
\end{equation}
which gives us the following result:
\begin{corollary}
Let $Y_1, \dots , Y_m$ be Brieskorn homology spheres and let $p$ be a prime such that for each $i$, the $\mathbb{Z}_p$-action on $Y_i$ is not free. Let $Y = Y_1 \# \cdots \# Y_m$ be the equivariant connected sum. Suppose that $W$ is a rational homology $4$-ball bounding $Y$ and that $p$ does not divide the order of $H^2(W ; \mathbb{Z})$. If ${\rm rk}( HF_{red}^+(Y_i)) > p \, {\rm rk}( HF_{red}^+(Y_i/\mathbb{Z}_p) )$ for some $i$, then the $\mathbb{Z}_p$-action on $Y$ does not extend smoothly to $W$.
\end{corollary}

Corollary \ref{cor:ne2} and (\ref{equ:dest}) also give us a non-extension result over positive-definite $4$-manifolds:
\begin{corollary}
Let $Y_1, \dots , Y_m$ be Brieskorn homology spheres and let $p$ be a prime such that for each $i$, the $\mathbb{Z}_p$-action on $Y_i$ is not free. Let $Y = Y_1 \# \cdots \# Y_m$ be the equivariant connected sum. Suppose that $W$ is a smooth, compact, oriented, $4$-manifold bounding $Y$, with positive definite intersection form and with $b_1(W) = 0$. If
\[
\delta(Y) + \sum_{k=1}^m  \left( {\rm rk}( HF_{red}^+(Y_k)) - p \, {\rm rk}( HF_{red}^+(Y_k/\mathbb{Z}_p) ) \right) > 0
\]
then the $\mathbb{Z}_p$-action on $Y$ does not extend smoothly and homologically trivially to $W$.
\end{corollary}

\subsection{Structure of the paper}
The paper is structured as follows. In \textsection \ref{sec:esw} we recall the basic results on equivariant Seiberg--Witten--Floer cohomology from \cite{bh} and the associated knot concordance invariants. In \textsection \ref{sec:fdelta} we examine in great detail the spectral sequence relating equivariant and non-equivariant Floer cohomology and use this to deduce more refined information about the equivariant delta invariants. In \textsection \ref{sec:tab}, we study the Floer homology of Brieskorn spheres $\Sigma(a,b,c)$ and use this to compute the invariants $\theta^{(c)}(T_{a,b})$ leading to the proof of Theorem \ref{thm:theta1}. In \textsection \ref{sec:branched}, we consider the case where $\mathbb{Z}_p$ acts non-freely on $\Sigma(a_1, \dots , a_r)$ and prove Theorem \ref{thm:deltabranch}. Finally in \textsection \ref{sec:free}, we consider the case where $\mathbb{Z}_p$ acts freely on $\Sigma(a_1, \dots , a_r)$ and prove Theorems \ref{thm:free0}, \ref{thm:free1} and \ref{thm:free2}.

\section{Equivariant Seiberg--Witten--Floer cohomology and knot concordance invariants}\label{sec:esw}

\subsection{Seiberg--Witten--Floer cohomology}\label{sec:swf}

Let $Y$ be a rational homology $3$-sphere and $\mathfrak{s}$ a spin$^c$-structure. For such a pair $(Y,\mathfrak{s})$, Manolescu constructed an $S^1$-equivariant stable homotopy type whose equivariant (co)homology groups are isomorphic to the Heegaard Floer or Monopole Floer (co)homology groups of $(Y,\mathfrak{s})$ \cite{man}. We denote the $S^1$-equivariant reduced cohomology groups with coefficients in $\mathbb{F}$ by $HSW^*(Y , \mathfrak{s} ; \mathbb{F})$ and refer to them as the {\em Seiberg--Witten--Floer cohomology} of $(Y , \mathfrak{s})$. If the coefficient group is understood then we will write $HSW^*(Y , \mathfrak{s})$. If $Y$ is an integral homology $3$-sphere, then it has a unique spin$^c$-structure and in this case we simply write $HSW^*(Y)$. We have that $HSW^*(Y , \mathfrak{s})$ is a graded module over the ring $H^*_{S^1} = H^*_{S^1}( pt ; \mathbb{F})$. Note that $H^*_{S^1} \cong \mathbb{F}[U]$, where $deg(U) = 2$.

There exists a chain of isomorphisms relating Seiberg--Witten--Floer homology to monopole Floer homology \cite{lima2} and to Heegaard Floer homology \cite{klt1,klt2,klt3,klt4,klt5,cgh1,cgh2,cgh3,tau}. In particular we have isomorphisms
\[
HSW^*(Y , \mathfrak{s}) \cong HF_+^*(Y , \mathfrak{s})
\]
where $HF_+^*(Y , \mathfrak{s})$ denotes the plus version of Heegaard Floer cohomology with coefficients in $\mathbb{F}$. Unless stated otherwise, we will take our coefficient group $\mathbb{F}$ to be a field. Then the universal coefficient theorem implies that the Heegaard Floer cohomology $HF_+^*(Y , \mathfrak{s})$ is isomorphic to the Heegaard Floer homology $HF^+_*(Y , \mathfrak{s})$, except that the action of $\mathbb{F}[U]$ on $HF^+_*(Y,\mathfrak{s})$ gets replaced by its dual, so $deg(U) = 2$ whereas in Heegaard Floer homology one has $deg(U)=-2$. We will frequently identify $HSW^*(Y , \mathfrak{s})$ with $HF^+_*(Y , \mathfrak{s})$, equipped with the dual $\mathbb{F}[U]$-module structure.

Let $d(Y , \mathfrak{s})$ denote the Ozsv\'ath--Szab\'o $d$-invariant. Due to the isomorphism $HSW^*(Y , \mathfrak{s}) \cong HF^+_*(Y , \mathfrak{s})$, we have that $d(Y , \mathfrak{s})$ is the minimal degree $i$ for which there exists an $x \in HSW^i(Y , \mathfrak{s})$ with $U^k x \neq 0$ for all $k \ge 0$. For notational convenience we define $\delta(Y , \mathfrak{s}) = d(Y,\mathfrak{s})/2$. 

We define $HSW_{red}^*(Y , \mathfrak{s}) = \{ x \in HSW^*(Y , \mathfrak{s}) \; | \; U^k x = 0 \text{ for some } k \ge 0\}$. More generally, given an $\mathbb{F}[U]$-module $M$, we write $M_{red}$ for the submodule of elements $x \in M$ such that $U^k x = 0$ for some $k \ge 0$.

Recall that for any $(Y , \mathfrak{s})$, there is an isomorphism of $\mathbb{F}[U]$-modules
\[
HF^+(Y , \mathfrak{s}) \cong \mathbb{F}[U]_{d(Y,\mathfrak{s})} \oplus HF^+_{red}(Y , \mathfrak{s}),
\]
where for any $\mathbb{F}[U]$-module $M^*$ and any $d \in \mathbb{Q}$, we define $M^*_d$ by $(M^i_d) = M^{i-d}$. It follows that we similarly have an isomorphism
\[
HSW^*(Y , \mathfrak{s}) \cong \mathbb{F}[U]_{d(Y , \mathfrak{s})} \oplus HSW_{red}(Y , \mathfrak{s}).
\]

\subsection{Equivariant Seiberg--Witten--Floer cohomology}\label{sec:eswfc}

Let $Y$ be a rational homology $3$-sphere and suppose that $\tau \colon Y \to Y$ an orientation preserving diffeomorphism of order $p$, where $p$ is prime. This gives an action of the finite group $G = \mathbb{Z}_p$ on $Y$ generated by $\tau$. Let $\mathfrak{s}$ be a spin$^c$-structure preserved by $\tau$. In \cite{bh}, the authors constructed the equivariant Seiberg--Witten--Floer cohomology groups $HSW^*_G(Y , \mathfrak{s})$. Except where stated otherwise, we take Floer cohomology with respect to the coefficient field $\mathbb{F} = \mathbb{Z}_p$. Then $HSW^*_G(Y , \mathfrak{s})$ is a module over the ring $H^*_{S^1 \times G} = H^*_{S^1 \times G}( pt ; \mathbb{F})$. If $p=2$, then $H^*_{S^1 \times G} \cong \mathbb{F}[U,Q]$, where $deg(U) = 2$, $deg(Q) = 1$. If $p$ is odd, then $H^*_{S^1 \times G} \cong \mathbb{F}[U,R,S]/(R^2)$, where $deg(U) = 2$, $deg(R)=1$, $deg(S)=2$. As in the non-equivariant case, the grading on $HSW^*_G(Y , \mathfrak{s})$ can in general take rational values. However, if $Y$ is an integral homology sphere then the grading is integer-valued.

The localisation theorem in equivariant cohomology implies that the localisation $U^{-1} HSW^*_G(Y , \mathfrak{s})$ is a free $U^{-1}H^*_{S^1 \times G}$-module of rank $1$. Letting $\mu$ denote a generator of $U^{-1} HSW^*_G(Y , \mathfrak{s})$, we have an isomorphism of the form
\[
\iota \colon U^{-1}HSW^*_G(Y , \mathfrak{s}) \to \mathbb{F}[Q , U , U^{-1}]\mu
\]
for $p=2$ and
\[
\iota \colon U^{-1}HSW^*_G(Y , \mathfrak{s}) \to \frac{\mathbb{F}[R , S , U , U^{-1}]}{(R^2)} \mu
\]
for $p$ odd. Following \cite[\textsection 3]{bh} we define a sequence of equivariant $\delta$-invariants as follows. The cases $p=2$ and $p \neq 2$ need to be treated separately. First suppose $p=2$. For each $j \ge 0$, we define $\delta^{(p)}_{j}(Y , \mathfrak{s} , \tau)$ to be $i/2 - j/2$, where $i$ is the least degree for which there exists an element $x \in HSW^i_G(Y , \mathfrak{s})$ and a $k \in \mathbb{Z}$ such that
\[
\iota x = Q^{j} U^k \mu \; ({\rm mod} \; Q^{j+1}).
\]
If $p \neq 2$, then for each $j \ge 0$, we define $\delta_{j}^{(p)}(Y , \mathfrak{s} , \tau)$ to be $i/2 - j$, where $i$ is the least degree for which there exists an element $x \in HSW^i_G(Y , \mathfrak{s})$ and a $k \in \mathbb{Z}$ such that
\[
\iota x = S^{j} U^k \mu \; ({\rm mod} \; S^{j+1}, RS^j).
\]
When the diffeomorphism $\tau$ is understood we will omit it from the notation and simply write the delta invariants as $\delta_j^{(p)}(Y , \mathfrak{s})$.

Various properties of the $\delta$-invariants are shown in \cite{bh}. In particular, we have:
\begin{itemize}
\item[(1)]{$\delta^{(p)}_{0}(Y,\mathfrak{s}) \ge \delta(Y,\mathfrak{s})$, where $\delta(Y,\mathfrak{s}) = d(Y,\mathfrak{s})/2$ and $d(Y,\mathfrak{s})$ is the Ozsv\'ath--Szab\'o $d$-invariant.}
\item[(2)]{$\delta_{j+1}^{(p)}(Y,\mathfrak{s}) \le \delta_{j}^{(p)}(Y , \mathfrak{s})$ for all $j \ge 0$.}
\item[(3)]{The sequence $\{ \delta_{j}^{(p)}(Y,\mathfrak{s}) \}_{j \ge 0}$ is eventually constant.}
\end{itemize}

Using property (3), we may define two additional invariants of $(Y , \mathfrak{s} , \tau)$ as follows. We define $\delta_\infty^{(p)}(Y , \mathfrak{s},\tau) = \lim_{j \to \infty} \delta_j^{(p)}(Y , \mathfrak{s},\tau)$ and we define $j^{(p)}(Y , \mathfrak{s},\tau)$ to be the smallest $j$ such that $\delta_j^{(p)}(Y , \mathfrak{s},\tau) = \delta_\infty^{(p)}(Y , \mathfrak{s},\tau)$. If $\tau$ is understood we will simply write $\delta_\infty^{(p)}(Y , \mathfrak{s})$ and $j^{(p)}(Y,\mathfrak{s})$.

\subsection{Knot concordance invariants}\label{sec:kci}

Given a knot $K \subset S^3$ and a prime number $p$, we let $Y = \Sigma_p(K)$ denote the degree $p$ cyclic cover of $S^3$ branched over $K$. Then $Y$ is a rational homology $3$-sphere \cite[Corollary 3.2]{liv} and it comes equipped with a natural $\mathbb{Z}_p$-action. Let $\pi \colon Y \to S^3$ denote the covering map. From \cite[Corollary 2.2]{jab}, any spin$^c$-structure on $Y \setminus \pi^{-1}(K)$ uniquely extends to $Y$. Then since $H^2( S^3 \setminus K ; \mathbb{Z}) = 0$, there is a unique spin$^c$-structure on $S^3 \setminus K$. The pullback of this spin$^c$-structure under $\pi$ extends uniquely to a spin$^c$-structure on $Y$. Following \cite{jab}, we denote this spin$^c$-structure by $\mathfrak{s}_0 = \mathfrak{s}_0(K , p)$. Uniqueness of the extension implies that $\mathfrak{s}_0$ is $\mathbb{Z}_p$-invariant. Thus for any prime $p$ and any $j \ge 0$, we obtain a knot invariant $\delta^{(p)}_j(K)$ by setting
\[
\delta_j^{(p)}(K) = 4 \delta_j^{(p)}( \Sigma_p( K ) , \mathfrak{s}_0 ).
\]

The invariants $\delta^{(p)}_j(K)$ are knot concordance invariants and they satisfy a number of properties \cite{bh}, \cite{bar}. In particular, for $p=2$ we have:
\begin{itemize}
\item[(1)]{$\delta^{(2)}_0(K) \ge \delta^{(2)}(K)$, where $\delta^{(2)}(K)$ is the Manolescu--Owens invariant \cite{mo}.}
\item[(2)]{$\delta^{(2)}_{j+1}(K) \le \delta^{(2)}_j(K)$ for all $j \ge 0$.}
\item[(3)]{$\delta^{(2)}_j(K) \ge -\sigma(K)/2$ for all $j \ge 0$ and $\delta^{(2)}_j(K) = -\sigma(K)/2$ for $j \ge g_4(K)-\sigma(K)/2$.}
\end{itemize}

Here $g_4(K)$ is the slice genus of $K$ and $\sigma(K)$ is the signature. Similarly for $p \neq 2$, we have:
\begin{itemize}
\item[(1)]{$\delta^{(p)}_0(K) \ge \delta^{(p)}(K)$, where $\delta^{(p)}(K)$ is the Jabuka invariant \cite{jab}.}
\item[(2)]{$\delta^{(p)}_{j+1}(K) \le \delta^{(p)}_j(K)$ for all $j \ge 0$.}
\item[(3)]{$\delta^{(p)}_j(K) \ge -\sigma^{(p)}(K)/2$ for all $j \ge 0$ and $\delta^{(p)}_j(K) = -\sigma^{(p)}(K)/2$ for $2j \ge (p-1)g_4(K)-\sigma^{(p)}(K)/2$.}
\end{itemize}

Here $\sigma^{(p)}(K)$ is defined as
\[
\sigma^{(p)}(K) = \sum_{j=1}^{p-1} \sigma_{K}( e^{2\pi i j/p}),
\]
where $\sigma_K(\omega)$ is the Levine--Tristram signature of $K$. Define $j^{(p)}(K)$ to be the smallest $j$ such that $\delta^{(p)}_j(K)$ attains its minimum. Thus
\[
j^{(p)}(K) = j^{(p)}( \Sigma_p(K) , \mathfrak{s}_0).
\]
From property (3), we see that the minimum value of $\delta^{(2)}_j(K)$ is precisely $-\sigma(K)/2$ and that
\[
j^{(2)}(K) \le g_4(K) - \sigma(K)/2.
\]
This can be re-arranged into a lower bound for the slice genus:
\[
g_4(K) \ge j^{(2)}(K) + \frac{\sigma(K)}{2}.
\]
Similarly for $p \neq 2$, the minimum value of $\delta^{(p)}_j(K)$ is $-\sigma^{(p)}(K)/2$ and that
\[
2 j^{(p)}(K) \le (p-1)g_4(K) - \frac{\sigma^{(p)}(K)}{2}.
\]
This can be re-arranged to
\[
g_4(K) \ge \frac{2 j^{(p)}(K)}{(p-1)} + \frac{\sigma^{(p)}(K)}{2(p-1)}.
\]
Replacing $K$ by its mirror $-K$, we  obtain the following slice genus bounds:
\[
g_4(K) \ge j^{(2)}(-K) - \frac{\sigma(K)}{2}
\]
for $p=2$, and 
\[
g_4(K) \ge \frac{2 j^{(p)}(-K)}{(p-1)} - \frac{\sigma^{(p)}(K)}{2(p-1)}
\]
for $p \neq 2$. Following \cite{bar}, we define knot concordance invariants $\theta^{(p)}(K)$ for all primes $p$ by setting
\[
\theta^{(2)}(K) = \max\left\{ 0 , j^{(2)}(-K) - \frac{\sigma(K)}{2} \right\}
\]
for $p=2$, and
\[
\theta^{(p)}(K) = \max\left\{ 0 , \frac{2 j^{(p)}(-K)}{(p-1)} - \frac{\sigma^{(p)}(K)}{2(p-1)} \right\}
\]
for $p \neq 2$. Then we have the slice genus bounds $g_4(K) \ge \theta^{(p)}(K)$ for all $p$.

\section{Further properties of the delta invariants}\label{sec:fdelta}

In this section we will use the spectral sequence of \cite{bh} relating equivariant and non-equivariant Seiberg--Witten--Floer cohomologies in order to deduce more precise information on the $\delta$-invariants. In particular, we will be able to determine the value of $j^{(p)}(Y)$ under certain assumptions on the Floer homology of $Y$.

\subsection{Spectral sequence}\label{sec:ss}

Let $Y$ be a rational homology $3$-sphere and $\tau \colon Y \to Y$ an orientation preserving diffeomorphism of prime order $p$. Let $\mathfrak{s}$ be a spin$^c$-structure preserved by $\tau$. Let the coefficient group be $\mathbb{F} = \mathbb{Z}_p$ and let $G = \langle \tau \rangle = \mathbb{Z}_p$.

Recall that $H^*_G$ is isomorphic to $\mathbb{F}[Q]$, where $deg(Q) = 1$ if $p=2$ and is isomorphic to $\mathbb{F}[R,S]/(R^2)$ where $deg(R) = 1$, $deg(S) = 2$ if  $p$ is odd. Thus in either case we have $H^{ev}_G \cong \mathbb{F}[S]$, where in the $p=2$ case we set $S = Q^2$.

Let $\{ E_r^{p,q} , d_r \}$ denote the spectral sequence for the $G$-equivariant Seiberg--Witten--Floer cohomology $HSW^*_{G}(Y , \mathfrak{s})$ \cite[Theorem 3.2]{bh}. In detail, this means there is a filtration $\{ \mathcal{F}_j^* \}_{j \ge 0}$ on $HSW^*_G(Y , \mathfrak{s})$ such that $E_\infty$ is isomorphic to the associated graded group 
\[
E_\infty \cong Gr( HSW^*_{G}(Y , \mathfrak{s}) ) = \bigoplus_{j \ge 0} \mathcal{F}^*_j/\mathcal{F}^*_{j+1}
\]
as $H^*_{S^1 \times G}$-modules. In terms of bigrading this means that $E^{p,q}_\infty \cong \mathcal{F}_p^{p+q}/\mathcal{F}_{p+1}^{p+q}$. Furthermore, we have that
\[
E_2^{p,q} = H^p( \mathbb{Z}_p ; HSW^q(Y , \mathfrak{s})).
\]

\subsection{Behaviour of the delta invariants}\label{sec:beh}

Let $(Y , \mathfrak{s} , \tau)$ be as in Section \ref{sec:ss}. We will examine the spectral sequence $\{ E_r^{p,q} \}$ to deduce some properties of the delta invariants $\{ \delta_j^{(p)}(Y,\mathfrak{s}) \}$.

To simplify notation, we will set $H^i = HSW^i(Y , \mathfrak{s})$ and $\widehat{H}^i = HSW^i_G(Y , \mathfrak{s})$. Let $d = d(Y , \mathfrak{s})$ and $\delta = d/2$. We will make the assumption that $H^i = 0$ unless $i = d \; ({\rm mod} \; 2)$. This assumption is satisfied if $Y$ is a Seifert homology sphere oriented so that $-Y$ is the boundary of a negative definite plumbing.

\begin{lemma}\label{lem:even}
If $p=2$ and if $HSW^*(Y , \mathfrak{s})$ is concentrated in degrees equal to $d(Y,\mathfrak{s})$ mod $2$, then for all $j \ge 0$ we have $\delta_{2j+1}^{(2)}(Y , \mathfrak{s}) = \delta_{2j}^{(2)}(Y ,\mathfrak{s})$.
\end{lemma}
\begin{proof}
Adapting the proof of \cite[Lemma 5.7]{bh} to the $p=2$ case, one can show that $Q \colon E_\infty^{p,q} \to E_\infty^{p+1,q}$ is surjective for all $p,q$. This implies that $Q \colon \mathcal{F}_j \to \mathcal{F}_{j+1}$ is surjective. Now let $x \in HSW_{\mathbb{Z}_2}^i( Y , \mathfrak{s})$ satisfy $\iota x = Q^{2j+1} U^k \mu \; ({\rm mod} \; Q^{2j+2})$ for some $k$ and assume that $x$ has the minimal possible degree, so $$\delta_{2j+1}^{(2)}(Y , \mathfrak{s}) = deg(x)/2 - (2j+1)/2 = i/2 - j - 1/2.$$ Note that since $deg(Q) = 1$, $deg(U) = 2$, we have that $deg(x) = 1 + d(Y,\mathfrak{s}) \; ({\rm mod} \; 2)$. This means that the image of $x$ in $HSW_{\mathbb{Z}_2}^i( Y , \mathfrak{s})/\mathcal{F}_1 = E_\infty^{0,q}$ is zero, because $E_\infty^{0,q} \subseteq E_2^{0,q} = H^0( \mathbb{Z}_2 ; HSW^q(Y , \mathfrak{s}))$ is concentrated in degrees equal to $d(Y , \mathfrak{s})$ mod $2$. Hence $x \in \mathcal{F}_1$. This implies that $x = Qy$ for some $y \in HSW^{i-1}_{\mathbb{Z}_2}(Y , \mathfrak{s})$. Then since $x =Qy$, we have
\[
\iota x = Q \iota y = Q^{2j+1} U^k \mu \; ({\rm mod} \; Q^{2j+2}).
\]
Moreover, $Q$ is injective on $U^{-1} HSW_{\mathbb{Z}_2}^*(Y , \mathfrak{s}) \cong \mathbb{F}[Q , U , U^{-1}]$, so we deduce that
\[
\iota y = Q^{2j} U^k \mu \; ({\rm mod} \; Q^{2j+1}).
\]
Therefore, from the definition of $\delta_{2j}^{(2)}(Y , \mathfrak{s})$, it follows that 
\[
\delta_{2j}^{(2)}(Y , \mathfrak{s}) \le deg(y)/2 - j = deg(x)/2 - 1/2 - j = \delta_{2j+1}^{(2)}(Y , \mathfrak{s}).
\]
But we always have $\delta_{2j+1}^{(2)}(Y , \mathfrak{s}) \le \delta_{2j}^{(2)}(Y , \mathfrak{s})$, so we must have an equality $\delta_{2j}^{(2)}(Y , \mathfrak{s}) = \delta_{2j+1}^{(2)}(Y , \mathfrak{s})$.
\end{proof}

\begin{remark}\label{rem:j2}
By Lemma \ref{lem:even}, if $p=2$ and $HSW^*(Y , \mathfrak{s})$ is concentrated in degrees equal to $d(Y,\mathfrak{s})$ mod $2$, then $j^{(2)}(Y , \mathfrak{s})$ is even.
\end{remark}

In what follows, we find it more convenient to consider not the whole equivariant Seiberg--Witten--Floer cohomology but only the part concentrated in degrees equal to $d(Y , \mathfrak{s})$ mod $2$. For this reason we consider even counterparts of the relevant ingredients, such as the spectral sequence and filtration. One advantage is that it allows us to treat the $p=2$ and $p \neq 2$ cases simultaneously.

Let $H^{ev}_{S^1 \times G}$ denote the subring of $H^*_{S^1 \times G}$ given by elements of even degree. Then $H^{ev}_{S^1 \times G} \cong \mathbb{F}[U,S]$. Let $\widehat{H}^{ev}$ denote the $H^{ev}_{S^1 \times G}$-submodule of $\widehat{H}^*$ given by elements of degree equal to $d$ mod $2$. The filtration $\{ \mathcal{F}_j^* \}$ defines a filtration $\{ \mathcal{F}_j^{ev} \}$ on $\widehat{H}^{ev}$, where $\mathcal{F}_j^{ev} = \mathcal{F}_{2j} \cap \widehat{H}^{ev}$. The associated graded $H^{ev}_{S^1 \times G}$-module of this filtration is isomorphic to 
\[
E_\infty^{ev} = \bigoplus_{p,q} E_\infty^{2p,q}.
\]

By \cite[Lemma  5.7]{bh}, the map $S \colon E_\infty^{2p,q} \to E_\infty^{2p+2 , q}$ is surjective for all $p,q$. It follows that $S \colon \mathcal{F}^{ev}_j \to \mathcal{F}^{ev}_{j+1}$ is surjective for each $j \ge 0$. In particular, since $\mathcal{F}^{ev}_0 = \widehat{H}^{ev}$, we see that $\mathcal{F}^{ev}_j = S^j \widehat{H}^{ev}$.

Recall the localisation isomorphism
\[
\iota \colon U^{-1} \widehat{H}^* \cong H^*_G[U , U^{-1}] \mu
\]
for some $\mu$. Restricting to $\widehat{H}^{ev}$, we get a localisation isomorphism
\[
\iota \colon U^{-1} \widehat{H}^{ev} \cong H^{ev}_G[U,U^{-1}]\mu \cong \mathbb{F}[S,U,U^{-1}]\mu.
\]
Let $\delta_j$ denote $\delta_j^{(p)}(Y,\mathfrak{s})$ if $p$ is odd and $\delta_{2j}^{(2)}(Y , \mathfrak{s})$ if $p=2$. Since $S = Q^2$ in the $p=2$ case, it follows that regardless of whether $p$ is even or odd, we have $\delta_j = a/2 - j$, where $a$ is the least degree such that there exists an element $x \in \widehat{H}^{a}$ with 
\[
\iota x = S^j U^k \mu \; ({\rm mod} \; S^{j+1})
\]
for some $k \in \mathbb{Z}$.

Observe that $E_2^{0,q} = H^0( \mathbb{Z}_p ; H^q )$ is a submodule of $H^q$. Then since $E_{r+1}^{0,q}$ is a submodule of $E_{r}^{0,q}$ for each $r \ge 2$, we see that $E_\infty^{0,*} = \mathcal{F}^{ev}_0/\mathcal{F}^{ev}_1$ can be identified with an $\mathbb{F}[U]$-submodule of $H^*$. We will denote this submodule by $J^* \subset H^*$. Now since $H^* \cong \mathbb{F}[U] \lambda \oplus H^*_{red}$, where $\lambda$ has degree $d$ and $H^*_{red} = \{ x \in H^* \; | \; U^m x = 0 \text{ for some } m \ge 0\}$. It follows that $J^* \cong \mathbb{F}[U]\theta \oplus J^*_{red}$, where $\theta$ has degree at least $d$ and $J^*_{red} \subseteq H^*_{red}$. Since $U^m \theta \neq 0$ for all $m \ge 0$, we see that the image of $\theta$ under the localisation map $\iota$ is non-zero. More precisely,
\[
\iota \theta = U^{\alpha} \mu \; ({\rm mod} \; S )
\]
for some $\alpha \in \mathbb{Z}$. On the other hand, any $x \in J^*$ has $U^m x = 0 \; ({\rm mod} \; \mathcal{F}_1^{ev})$ for some $m \ge 0$. This means that $U^m x = Sy$ for some $y$ and thus $\iota x$ is a multiple of $S$. It follows that $deg(\theta) = \delta_0$.

By \cite[Proposition 3.14]{bh}, we then have that $\delta_j = a/2 - j$, where $a$ is the least degree such that there exists an element $x \in \widehat{H}^a$ such that
\[
U^m x = S^j U^k \theta \; ({\rm mod} \; \mathcal{F}_{j+1}^{ev})
\]
for some $m, k \ge 0$.

Consider again the sequence $\delta_0 \ge \delta_1 \ge \delta_2 \ge \cdots$. This sequence is eventually constant, hence there exists a finite set of indices $0 < j_1 < j_2 < \cdots < j_r$ such that for each $j > 0$, $\delta_j < \delta_{j-1}$ if and only if $j = j_i$ for some $i$. We let $n_i = \delta_{j_{i-1}} - \delta_{j_i} > 0$ for $1 < i \le r$ and $n_1 = \delta_{j_1} - \delta_0$ for $i=1$. Thus $\delta_{j_i} = \delta_0 - (n_1 + \cdots + n_i )$.

\begin{lemma}\label{lem:indep}
There exists $x_1, \dots , x_r \in J^*_{red}$ such that:
\begin{itemize}
\item[(1)]{The set $\{ U^a x_i \; | \; 1 \le i \le r, 0 \le a \le n_i -1 \}$ is linearly independent over $\mathbb{F}$. In particular, $U^a x_i \neq 0$ for $a < n_i$.}
\item[(2)]{We have $\delta_{j_i} = deg(x_i)/2 - j_i$.}
\item[(3)]{Each $x_i$ has a lift to an element $y_i \in \widehat{H}^i$ satisfying
\[
U^{m_i} y_i = S^{j_i} U^{k_i} \theta \; ({\rm mod} \; \mathcal{F}_{j_i+1}^{ev})
\]
for some $m_i, k_i \ge 0$.}
\end{itemize}
\end{lemma}
\begin{proof}
By the definition of $\delta_{j_i}$, there exists $y_i \in \widehat{H}^*$ such that $\delta_{j_i} = deg(y_i)/2 - j_i$ and $U^{m_i} y_i = S^{j_i} U^{k_i} \theta \; ({\rm mod} \; \mathcal{F}_{j_i+1}^{ev})$ for some $m_i,k_i$. Let $x_i \in H^*$ be the image of $y_i$ under the natural map $\widehat{H}^{ev} \to \widehat{H}^{ev}/\mathcal{F}_1^{ev} \cong J^*$. We claim that $x_i \in J^*_{red}$. If not, then $x_i = c U^k \theta + w$ for some $c \in \mathbb{F}^*$, $k \ge 0$ and $w \in J^*_{red}$. Then $U^m x_i = c U^{k+m} \theta$ for some $m, k \ge 0$ and hence $U^m y_i = c U^{k+m} \theta \; ({\rm mod} \; \mathcal{F}^{ev}_1)$. This contradicts $U^{m_i} y_i = S^{j_i} U^{k_i} \theta \; ({\rm mod} \; \mathcal{F}_{j_i+1}^{ev})$ as $S^{j_i} U^{k_i} \theta = 0 \; ({\rm mod} \; \mathcal{F}^{ev}_1)$ (since $j_i > 0$). So $x_i \in J^*_{red}$. It is now evident that (2) and (3) are satisfied. We claim that (1) also holds, that is, the set $\{ U^a x_i \; | \; 1 \le i \le r, 0 \le a \le n_i -1 \}$ is linearly independent over $\mathbb{F}$.

Suppose we have a non-trivial linear relation amongst elements of $\mathcal{B} = \{ U^a x_i \; | \; 1 \le i \le r, 0 \le a \le n_i -1 \}$ of degree $s$. So
\[
\sum_i c_i U^{\beta_i} x_i = 0
\]
for some $c_i \in \mathbb{F}$, not all zero and $\beta_i = (s - deg(x_i) )/2$. Furthermore, the sum is restricted to those $i$ such that $0 \le \beta_i \le n_i - 1$ (so that $U^{\beta_i} x_i$ belongs to $\mathcal{B}$). Lifting this to $\widehat{H}^{ev}$, we have $\sum_i c_i U^{\beta_i} y_i \in \mathcal{F}^{ev}_1$ and hence
\begin{equation}\label{equ:w}
\sum_i c_i U^{\beta_i} y_i = Sw
\end{equation}
for some $w \in \widehat{H}^{s-2}$. Let $u$ be the smallest value of $i$ such that $c_u \neq 0$. Multiplying by a sufficiently large power of $U$ and taking the result modulo $\mathcal{F}^{ev}_{j_u+1}$, we see that
\[
S U^m w = c_i S^{j_u} U^{b} \theta \; ({\rm mod} \; \mathcal{F}^{ev}_{j_u+1})
\]
for some $m,b \ge 0$. Hence $S \iota (U^m w) = S \iota (c_i S^{j_u-1} U^b \theta) \; ({\rm mod} \; S^{j_u+1})$, where $\iota$ is the localisation map. Since $S$ is injective on $\mathbb{F}[S,U,U^{-1}]\mu$, we also have $\iota w = c_i S^{j_u-1} U^{b-m} \iota \theta \; ({\rm mod} \; S^{j_u})$. Now we recall that $\iota \theta = U^\alpha \mu \; ({\rm mod} \; S)$ for some $\alpha$, hence $\iota (c_i^{-1} w ) = S^{j_u-1} U^{b-m+\alpha} \mu \; ({\rm mod} \; S^{j_u})$. By the definition of $\delta_{j_u-1}$, we must therefore have $deg(w)/2 \ge \delta_{j_u-1} + (j_u-1)$. But $deg(w) = s-2$, so $s \ge \delta_{j_u-1} + j_u$. On the other hand, equating degrees in Equation (\ref{equ:w}) gives $s = \beta_u + deg(x_u)/2 = \beta_u + \delta_{j_u} + j_u$. This gives
\[
\beta_u = s - \delta_{j_u} - j_u \ge \delta_{j_u-1} - \delta_{j_u} = n_u.
\]
But this contradicts $\beta_u \le n_u-1$. So no such linear relation can exist.
\end{proof}

\begin{corollary}\label{cor:ineq}
We have that $\delta_0 - \delta_j \le dim_{\mathbb{F}}( J_{red}^*) \le dim_{\mathbb{F}}( HF_{red}^*(Y , \mathfrak{s}) )$ for all $j \ge 0$.
\end{corollary}
\begin{proof}
Since $\delta_j = \delta_{j_r}$ for all $j \ge j_r$, we have that $\delta_0 - \delta_j \le \delta_0 - \delta_{j_r}$ for all $j$. But the subspace $span_{\mathbb{F}}\{ U^a x_i \; | \; 1 \le i \le r, 0 \le a \le n_i -1 \} \subseteq J^*$ of Lemma \ref{lem:indep} has dimension $n_1 + \cdots + n_r = \delta_0 - \delta_{j_r}$.
\end{proof}

Let $j'(Y,\mathfrak{s})$ denote the smallest $j$ such that $\delta_j$ attains its minimum. Thus $j'(Y,\mathfrak{s}) = j^{(p)}(Y,\mathfrak{s})$ if $p$ is odd and $j'(Y,\mathfrak{s}) = j^{(2)}(Y,\mathfrak{s})/2$ if $p=2$, by Remark \ref{rem:j2}.

\begin{proposition}\label{prop:jinv}
Let $Y$ be a rational homology $3$-sphere, $\tau \colon Y \to Y$ an orientation preserving diffeomorphism of order $p$, where $p$ is prime. Let $\mathfrak{s}$ be a spin$^c$-structure preserved by $\tau$. Suppose that the following conditions hold:
\begin{itemize}
\item[(1)]{$HF^+(Y , \mathfrak{s})$ is non-zero only in degrees $i = d(Y , \mathfrak{s}) \; ({\rm mod} \; 2)$.}
\item[(2)]{$\delta_0 - \delta_j = dim_{\mathbb{F}}( HF_{red}^*(Y , \mathfrak{s}) )$ for some $j \ge 0$.}
\item[(3)]{$HF_{red}^*(Y , \mathfrak{s}) \neq 0$.}
\item[(4)]{Let $\ell^+(Y,\mathfrak{s})$ denote the highest non-zero degree in $HF^+_{red}(Y,\mathfrak{s})$. Any non-zero element in the image of $U \colon HF_{red}^+(Y , \mathfrak{s}) \to HF_{red}^+(Y , \mathfrak{s})$ has degree strictly less than $\ell^+(Y,\mathfrak{s})$.}
\end{itemize}

Then
\[
j'(Y,\mathfrak{s}) = \ell^+(Y,\mathfrak{s})/2 - \delta_0(Y,\mathfrak{s}) + dim_{\mathbb{F}}( HF_{red}^*(Y , \mathfrak{s}) ).
\]
Moreover, if $Y$ is an integral homology sphere and $\delta(Y) \in 2 \mathbb{Z}$, then
\[
j'(Y) = \ell^+(Y)/2 - \delta_0(Y) + \delta(Y) + \lambda(Y),
\]
where $\lambda(Y)$ is the Casson invariant of $Y$.
\end{proposition}
\begin{proof}
Let $x_1, \dots , x_r \in J_{red}^*$ be as in Lemma \ref{lem:indep}. It follows that $j'(Y,\mathfrak{s}) = j_r$. Choose $j$ such that $\delta_0 - \delta_j = dim_{\mathbb{F}}( HF_{red}^*(Y , \mathfrak{s}) )$. From Corollary \ref{cor:ineq}, we have 
\[
dim_{\mathbb{F}}( HF_{red}^*(Y , \mathfrak{s}) ) = \delta_0 - \delta_j \le \delta_0 - \delta_{j_r} = n_1 + \cdots + n_r \le  dim_{\mathbb{F}}( J_{red}^*) \le dim_{\mathbb{F}}( HF_{red}^*(Y , \mathfrak{s}) ).
\]
It follows that we must have equalities throughout. So $\delta_0 - \delta_{j_r} = n_1 + \cdots + n_r = dim_{\mathbb{F}}( HF_{red}^*(Y , \mathfrak{s}) )$. So $\{ U^a x_i \; | \; 1 \le i \le r, 0 \le a \le n_i -1 \}$ is a basis for $HF_{red}^*(Y,\mathfrak{s})$. Set $\alpha_i = deg(x_i)/2$. Lemma \ref{lem:indep} (2) implies that $\delta_{j_i} = \alpha_i - j_i$ for $1 \le i \le r$. Setting $\alpha_0 = \delta_0$ and $j_0 = 0$, we also have $\delta_{j_i} = \alpha_i - j_i$ for $i=0$. Further, we have that $\delta_{j_i} - \delta_{j_{i-1}} = n_i$ for $1 \le i \le r$. Hence we obtain $(\alpha_i + n_i - 1) - \alpha_{i-1} = j_i - j_{i-1} - 1$. But $j_i > j_{i-1}$, so $j_i - j_{i-1} - 1 \ge 0$, giving
\begin{equation}\label{equ:alphain}
(\alpha_i + n_i - 1) \ge \alpha_{i-1}.
\end{equation}
Now we observe that $2(\alpha_i + n_i - 1)$ is the degree of $U^{n_i-1} x_i$ and $2\alpha_{i-1}$ is the degree of $\alpha_{i-1}$. Consider the space $A = HF^{\ell^+(Y,\mathfrak{s})}_{red}(Y , \mathfrak{s})$ of top degree elements in $HF^+_{red}(Y , \mathfrak{s})$. Let $a = dim_{\mathbb{F}}(A) > 0$. By assumption (4), no non-zero element of $A$ is in the image of $U$. Hence there must exist indices $k_1 < k_2 < \dots < k_a$ such that $x_{k_1}, \dots , x_{k_a}$ is a basis for $A$. In fact, $K = \{k_1, \dots , k_a\}$ is precisely the set of indices $i$ such that $2\alpha_i = \ell^+(Y , \mathfrak{s})$. Suppose $i \in K$ and $i < r$. Then (\ref{equ:alphain}) gives $2(\alpha_{i+1} + n_{i+1} - 1) \ge 2\alpha_{i} = \ell^+(Y,\mathfrak{s})$. But $2(\alpha_{i+1} + n_{i+1} - 1)$ is the degree of $U^{n_{i+1}-1} x_{i+1}$ and $2 \ell^+(Y , \mathfrak{s})$ is the highest degree in $HF^+_{red}(Y,\mathfrak{s})$. So we must have $2(\alpha_{i+1} + n_{i+1} - 1) = \ell^+(Y,\mathfrak{s})$. But from assumption (4), this can only happen if $n_{i+1} = 1$, hence $2\alpha_{i+1} = \ell^+(Y , \mathfrak{s})$ and so $i+1 \in K$. So if $i \in K$ and $i < r$, then $i+1 \in K$. It follows that $K = \{ r-a+1 , r-a +2 , \dots , r\}$. In particular, $r \in K$ and $\alpha_r = \ell^+(Y,\mathfrak{s})/2$. Therefore
\[
j'(Y , \mathfrak{s}) = j_r = \alpha_r - \delta_{j_r} = \ell^+(Y , \mathfrak{s})/2 - \delta_0(Y , \mathfrak{s}) + dim_{\mathbb{F}}( HF_{red}^*(Y , \mathfrak{s}) ),
\]
where the last equality holds since $\delta_0 - \delta_{j_r} = dim_{\mathbb{F}}( HF_{red}^*(Y , \mathfrak{s}) )$.

Lastly, suppose that $Y$ is an integral homology sphere and $\delta(Y) \in 2 \mathbb{Z}$. From assumption (1), it follows that $HF_{red}^+(Y)$ is non-zero only in even degrees. Then from \cite[Theorem 1.3]{os}, it follows that $ dim_{\mathbb{F}}( HF_{red}^*(Y)) = \delta(Y) + \lambda(Y)$, where $\lambda(Y)$ is the Casson invariant. Hence we have that $j'(Y) = \ell^+(Y)/2 - \delta_0(Y) + \delta(Y) + \lambda(Y)$.
\end{proof}
\subsection{Delta invariants of Brieskorn spheres}\label{sec:deltabri}

Given pairwise coprime integers $a_1 , \dots , a_r$ with $r \ge 3$ define the Brieskorn homology sphere $\Sigma(a_1 , \dots , a_r)$ to be the link of the singularity at the origin of the variey $\{ (z_1 , \dots , z_r ) \in \mathbb{C}^n \; | \; b_{i1} z_1^{a_1} + b_{i2}z_2^{a_2} + \cdots + b_{ir}z_r^{a_r} = 0, \; 1 \le i \le r-2 \}$, where $( b_{ij} )$ is a sufficiently generic $(r-2) \times r$ matrix. Then $\Sigma(a_1 , a_2 , \dots , a_r)$ is the Seifert homology sphere $M(e_0 , (a_1,b_1) , \dots , (a_r , b_r) )$, where $e_0 , b_1, \dots , b_r$ are uniquely determined by the conditions that $0 < b_j < a_j$ and 
\begin{equation}\label{equ:integral}
e_0 + \sum_{j=1}^r \frac{b_j}{a_j} = -\frac{1}{a_1 a_2 \cdots a_r}.
\end{equation}
Note that $e_0 < 0 $ because $b_j > 0$ for all $j$.

Let $Y = \Sigma(a_1 , \dots , a_r)$ be a Brieskorn homology sphere. The Seifert structure of $Y$ gives a circle action. For any prime $p$, the restriction of the circle action to $\mathbb{Z}_p \subset S^1$ defines an action of the finite cyclic group $G = \mathbb{Z}_p$. We are interested in studying the equivariant Seiberg--Witten--Floer cohomology groups of $Y$ and $-Y$ with respect to this action. Since $Y$ has a unique spin$^c$-structure we will omit it from the notation and write $HSW^*_G(Y)$ and $HSW^*_G(-Y)$. Similarly the $\delta$-invariants will be denoted $\delta_j^{(p)}(Y)$ and $\delta_j^{(p)}(-Y)$.

To understand the equivariant Seiberg--Witten--Floer cohomology groups of $Y$ and $-Y$, we will use the spectral sequence relating it to the corresponding non-equivariant groups. We have
\begin{align*}
HSW^*(Y)  \cong HF^+(Y) & \cong \mathbb{F}[U]_{d(Y)} \oplus HF_{red}^+(Y) \\
HSW^*(-Y) \cong HF^+(-Y) & \cong \mathbb{F}[U]_{-d(Y)} \oplus HF_{red}^+(-Y)
\end{align*}

From \cite{os2}, we have that $HF^+(-Y)$ is concentrated in even degrees. It follows that $d(Y)$ is even and $HF_{red}^+(Y)$ is concentrated in {\em odd} degrees. Furthermore $HF^+(-Y)$ can be computed from the graded roots algorithm \cite{nem}. This algorithm implies that $HF_{red}^+(-Y)$ is concentrated in degrees at least $-d(Y)$. Dually this implies that $HF_{red}^+(Y)$ is concentrated in degrees at most $d(Y)-1$.

\begin{proposition}\label{prop:deltaprop}
Let $Y = \Sigma(a_1, \dots , a_r)$ be a Brieskorn homology sphere. We have:
\begin{itemize}
\item[(1)]{${\rm rk}( HF_{red}^+(Y) ) = -\delta(Y) - \lambda(Y)$ where $\lambda(Y)$ is the Casson invariant of $Y$.}
\item[(2)]{$\delta_j^{(p)}(Y) = \delta_\infty^{(p)}(Y)$ for all $j \ge 0$.}
\item[(3)]{$\delta(Y) \le \delta_\infty^{(p)}(Y) \le -\lambda(Y)$ and $\lambda(Y) \le \delta_j^{(p)}(-Y) \le -\delta(Y)$ for all $j \ge 0$.}
\end{itemize}
\end{proposition}

\begin{proof}
From \cite[Theorem 1.3]{os} we have that $\chi( HF_{red}^+(Y) ) = \delta(Y) + \lambda(Y)$. But $HF_{red}^+(Y)$ is concentrated in odd degrees, which gives (1).

We will prove (2) and (3) in the case that $p$ is odd. The case $p=2$ is similar and omitted for brevity. Recall that for odd primes we have $H^*_{\mathbb{Z}_p} \cong \mathbb{F}[R,S]/(R^2)$ with $deg(R)=1$, $deg(S) = 2$.

As in Section \ref{sec:ss}, let $\{ E_r^{p,q}(Y) , d_r \}$ denote the spectral sequence for the equivariant Seiberg--Witten--Floer cohomology $HSW^*_{\mathbb{Z}_p}(Y)$. We have a filtration $\{ \mathcal{F}_j^*(Y) \}_{j \ge 0}$ on $HSW^*_{\mathbb{Z}_p}(Y)$ such that $E_\infty(Y)$ is isomorphic to the associated graded group 
\[
E_\infty(Y) \cong Gr( HSW^*_{\mathbb{Z}_p}(Y)) = \bigoplus_{j \ge 0} \mathcal{F}_j(Y)/\mathcal{F}_{j+1}(Y)
\]
as $H^*_{S^1 \times \mathbb{Z}_p}$-modules. Furthermore, we have that
\[
E_2^{p,q}(Y) = H^p( \mathbb{Z}_p ; HF^{+}_q(Y)).
\]
Now since the $\mathbb{Z}_p$-action on $Y$ is contained in a circle action, the generator of the $\mathbb{Z}_p$-action is smoothly isotopic to the identity and acts trivially on $HF^+(Y)$. Hence we further have
\[
E_2^{p,q}(Y) \cong HF^{+}_q(Y) \otimes_{\mathbb{F}} H^p_{\mathbb{Z}_p}.
\]
Choose an element $\theta \in HF^+_{d(Y)}(Y)$ such that $U^k \theta \neq 0$ for all $k \ge 0$. Then $HF^+(Y) \cong \mathbb{F}[U]\theta \oplus HF^+_{red}(Y)$ and hence
\begin{align*}
E_2(Y) &\cong H^*_{\mathbb{Z}_p}[U]\theta \oplus \left( HF_{red}^+(Y) \otimes_{\mathbb{F}} H^*_{\mathbb{Z}_p} \right) \\
& \cong \frac{\mathbb{F}[U,R,S]}{(R^2)} \theta \oplus \frac{HF_{red}^+(Y)[R,S]}{(R^2)}.
\end{align*}

We have a similar spectral sequence $E_r^{p,q}(-Y)$ and filtration $\{ \mathcal{F}_j(-Y) \}$ on $HSW^*_{\mathbb{Z}_p}(-Y)$ such that $E_\infty^{*,*}(-Y)$ is isomorphic to the associated graded $H^*_{S^1 \times \mathbb{Z}_p}$-module of the filtration. We have
\[
E_2(-Y) \cong \frac{\mathbb{F}[U,R,S]}{(R^2)} \omega \oplus \frac{HF_{red}^+(-Y)[R,S]}{(R^2)}
\]
for some $\omega$ of bi-degree $(0 , -d(Y))$. Now recall that $HF_{red}^+(-Y)$ is concentrated in degrees at least $-d(Y)$. Hence $d_r(\omega) = 0$ for all $r \ge 2$. It immediately follows that $\delta_0^{(p)}(-Y) = -\delta(Y)$ and hence $\delta_j^{(p)}(-Y) \le -\delta(Y)$ for all $j \ge 0$.

Define $E_r(Y)_{red}$ to be the subgroup of $E_r(Y)$ consisting of elements $x$ such that $U^k x = 0$ for some $k \ge 0$. So for $r = 2$ we have $E_2(Y)_{red} \cong HF_{red}^+(Y)[R,S]/(R^2)$ and 
\[
E_2(Y) \cong \frac{\mathbb{F}[U,R,S]}{(R^2)} \theta \oplus E_2(Y)_{red}.
\]

As explained in \cite[\textsection 5]{bh}, the image of the differential $d_r$ is contained in $E_r(Y)_{red}$. It follows that for each $r \ge 2$, $E_r(Y)$ is of the form
\[
E_r(Y) \cong \frac{\mathbb{F}[U,R,S]}{(R^2)} U^{m_r}\theta \oplus E_r(Y)_{red}
\]
for some increasing sequence $0 = m_0 \le m_1 \le \dots $. It is further shown in \cite[\textsection 5]{bh} that the sequence $m_r$ must be eventually constant. Hence we have
\[
E_\infty(Y) \cong \frac{\mathbb{F}[U,R,S]}{(R^2)} U^{m}\theta \oplus E_\infty(Y)_{red}
\]
where $m = \lim_{r \to \infty} m_r$.

Since we work over a field $\mathbb{F}$, it is possible to (non-canonically) split the filtration $\{ \mathcal{F}_j(Y) \}$ giving an isomorphism of $\mathbb{F}$-vector spaces
\[
HSW^*_{\mathbb{Z}_p}(Y) \cong E_\infty^{*,*}(Y).
\]
Under this isomorphism $\mathcal{F}_j(Y)$ corresponds to the subspace of $E_\infty(Y)$ spanned by homogeneous elements of bi-degree $(p,q)$ where $p \ge j$. We fix a choice of such an isomorphism $\varphi \colon HSW^*_{\mathbb{Z}_p}(Y) \to E_\infty^{*,*}(Y)$ and henceforth identify $HSW^*_{\mathbb{Z}_p}(Y)$ with $E_\infty^{*,*}(Y)$. This isomorphism will typically not be an isomorphism of $H^*_{S^1 \times \mathbb{Z}_p}$-modules. Nevertheless we can use the isomorphim $\varphi$ to induce a new $H^*_{S^1 \times \mathbb{Z}_p}$-module action on $E_\infty^{*,*}(Y)$ corresponding to the one on $HSW^*_{\mathbb{Z}_p}(Y)$. To be precise, if $c \in H^*_{S^1 \times \mathbb{Z}_p}$, then we define $\widehat{c} \colon E_\infty \to E_\infty$ by $\widehat{c} \, x = \varphi( c \varphi^{-1}(x) )$. Consider in particular $\widehat{U}$. The action of $\widehat{U}$ will typically not respect the bi-grading, but we can decompose it into homogeneous components as
\[
\widehat{U} = U_{(0,2)} + U_{(1,1)} + U_{(2,0)} + U_{(3,-1)} + \cdots
\]
where $U_{j , 2-j} \colon E_\infty^{p,q}(Y) \to E_\infty^{p+j , q+2-j}(Y)$. Note that there are only terms of bi-degree $(j,2-j)$ for $j \ge 0$ as $\widehat{U}$ respects the filtration $\{ \mathcal{F}_j(Y) \}$. Note also that $U_{(0,2)} = U$ because taking the associated graded module of the filtration recovers the original $H^*_{S^1 \times \mathbb{Z}_p}$-module structure on $E_\infty(Y)$.

Recall that $\delta_{j}^{(p)}(Y)$ is given by $i/2 - j$ where $i$ is the least degree such that there exists an $x \in HSW^i_{\mathbb{Z}_p}(Y)$ with $U^k x = S^j U^l \theta \; ({\rm mod} \; \mathcal{F}_{2j+1}(Y) )$ for some $k,l \ge 0$.

Recall that $E_\infty(Y) \cong \mathbb{F}[U,R,S]/(R^2) U^{m}\theta \, \oplus \, (E_\infty(Y))_{red}$. Since $\theta$ has bi-degree $(0, d(Y) )$, each homogeneous element in $ \mathbb{F}[U,R,S]/(R^2) U^{m}\theta$ has bi-degree $(p,q)$ where $q \ge d(Y) + 2m$. On the other hand, since $E_\infty(Y)_{red}$ is a subquotient of $HF_{red}^+(Y)[R,S]/(R^2)$ and $HF_{red}^+(Y)$ is concentrated in degrees at most $d(Y)-1$, we see that each homogeneous element $x$ in $E_\infty(Y)_{red}$ has bi-degree $(p,q)$, where $q \le d(Y)-1$. Then $U x = U_{(0,2)}x + U_{(1,1)}x + \cdots $ is a sum of homogeneous terms $U_{(j,2-j)}x$ of bidegree $(p+j , q + 2-j )$. Consider first $U_{(0,2)}x$. Since $U_{(0,2)} = U$, we see that $U_{(0,2)}x \in E_\infty(Y)_{red}$. Each of the remaining terms $U_{(j,2-j)} x$ for $j \ge 1$ has bi-degree of the form $(p',q') = (p+j , q+2-j)$, hence $q' = q+2-j \le q+1 \le d(Y)$. Thus if $m > 0$, then $U_{(j,2-j)}x$ must belong again to $E_\infty(Y)_{red}$. This implies that $E_\infty(Y)_{red}$ is an $\mathbb{F}[\widehat{U}]$-submodule of $E_\infty(Y)$. It follows easily from this that for any fixed $j \ge 0$, there exists a $k \ge 0$ such that $U^k E_\infty(Y)_{red} \subseteq \mathcal{F}_{2j+1}(Y)$. This further implies that $x = S^j U^m \theta$ is a minimal degree element such that $U^k x = S^j U^l \theta \; ({\rm mod} \; \mathcal{F}_{2j+1}(Y) )$ for some $k,l \ge 0$, hence $\delta_j^{(p)}(Y) = m + d(Y)/2$ for all $j \ge 0$. So $\delta_j^{(p)}(Y) = \delta_\infty^{(p)}(Y) = m + \delta(Y) > \delta(Y)$. From \cite[Proposition 5.10]{bh}, we also have that $m \le {\rm rk}( HF_{red}^+(Y) ) = -\delta(Y) - \lambda(Y)$, so $\delta_\infty^{(p)}(Y) = m + \delta(Y) \le -\lambda(Y)$.

If $m=0$, then we still have that $\widehat{U} E_\infty(Y)_{red} \subseteq E_\infty(Y)_{red}$, for if this were not the case then we would have some $x \in E_{\infty}(Y)_{red}^{a,d(Y)-1}$ such that $U_{(1,1)}x \notin E_{\infty}(Y)_{red}$. Further, since $HF_{red}^+(Y)$ is concentrated in degrees at most $d(Y)-1$, it follows that $U_{(0,2)}x = 0$ and that $U_{(1,1)}x = c \theta$ for some non-zero element $c \in H_{\mathbb{Z}_p}^{a+1}$. Then since $U_{(j,2-j)}x \in \mathcal{F}_{a+2}(Y)$ for $j \ge 2$, it follows that
\[
\widehat{U}x = c \theta \; ({\rm mod} \; \mathcal{F}_{a+2}(Y) ).
\]
Hence $\delta_{\mathbb{Z}_p , c }(Y) < \delta(Y)$. On the other hand, we have already shown that $\delta_{\mathbb{Z}_p , 1} = \delta_0^{(p)}(-Y) = -\delta(Y)$. Using this and \cite[Theorem 4.4]{bh}, we have
\[
0 \le \delta_{\mathbb{Z}_p , c}(Y) + \delta_{\mathbb{Z}_p , 1}(-Y) < \delta(Y) - \delta(Y) = 0,
\]
a contradiction. So even in the $m=0$ case we still have that $\widehat{U} E_{\infty}(Y)_{red} \subseteq E_{\infty}(Y)_{red}$ and so we again conclude that $\delta_j^{(p)}(Y) = \delta_\infty^{(p)}(Y)$ for all $j \ge 0$ and that $\delta_\infty^{(p)}(Y) \le -\lambda(Y)$. This proves (2).

We have proven everything except the inequality $\delta_j^{(p)}(-Y) \ge \lambda(Y)$. But from $\delta_0^{(p)}(Y) \le -\lambda(Y)$ and \cite[Theorem 4.4]{bh}, it follows that
\[
0 \le \delta_j^{(p)}(-Y) + \delta_0^{(p)}(Y) \le \delta_j^{(p)}(-Y) -\lambda(Y),
\]
for all $j \ge 0$. Hence $\delta_j^{(p)}(-Y) \ge \lambda(Y)$.
\end{proof}

\section{Computation of $\theta^{(c)}(T_{a,b})$}\label{sec:tab}

Let $a,b,c > 1$ be coprime integers and suppose also that $c$ is prime. In this section we will prove that $\theta^{(c)}(T_{a,b}) = (a-1)(b-1)/2$, where $T_{a,b}$ is the $(a,b)$-torus knot. Since $\Sigma_c( T_{a,b} ) = \Sigma(a,b,c)$, we will be interested in computing the invariant $j^{(c)}( \Sigma(a,b,c) )$ of the Brieskorn sphere $\Sigma(a,b,c)$. Consequently we are interested in studying the structure of the Floer homology of $\Sigma(a,b,c)$. Combined with the results of Section \ref{sec:fdelta}, we will be able to carry out the computation of $\theta^{(c)}(T_{a,b})$.

\subsection{Floer homology of $-\Sigma(a,b,c)$}\label{sec:abc}

Let $1 < a < b < c$ be pairwise coprime integers and let $\Sigma(a,b,c)$ denote the Brieskorn homology $3$-sphere oriented so that it is the link of the singularity $x^a + y^b + z^c = 0$ in $\mathbb{C}^3$. The graded roots algorithm of N\'emethi can be used to compute the Floer homology $HF^+(-\Sigma(a,b,c))$ \cite{nem}. In this case the algorithm produces a $\tau$ function $\tau \colon\mathbb{Z}_{\ge 0} \to \mathbb{Z}$ from which a graded root $(R ,\chi)$ may be constructed \cite[Example 3.4 (3)]{nem}. The graded root $(R , \chi)$ gives rise to an associated $\mathbb{F}[U]$-module $\mathbb{H}(R,\chi)$, which up to a grading shift coincides with $HF^+(-\Sigma(a,b,c))$.

For a Brieskorn sphere $\Sigma(a,b,c)$, the $\tau$ function $\tau(i)$ is given as follows \cite[\textsection 11]{nem}, \cite{caka}. Let $e_0 , p_1, p_2, p_3$ be the unique integers with $0 < p_1 < a$, $0 < p_2 < b$, $0 < p_3 < c$ and
\[
a b c e_0 + p_1 bc + a p_2 c + ab p_3 = -1.
\]
Then
\begin{equation}\label{equ:tau1}
\tau(i) = \sum_{n=0}^{i-1} \Delta(n),
\end{equation}
where
\begin{equation}\label{equ:tau2}
\Delta(n) = 1 - e_0 n - \left\lceil \frac{p_1 n}{a} \right\rceil - \left\lceil \frac{p_2 n}{b} \right\rceil - \left\lceil \frac{p_3 n}{c} \right\rceil.
\end{equation}
By \cite[Theorem 1.3]{caka}, we have $\tau(n+1) \ge \tau(n)$ for $n > N$, where
\[
N = abc - bc - ac - ab.
\]
Note that $N > 0$ except when $(a,b,c) = (2,3,5)$. We will exclude this case from the discussion unless stated otherwise. For the purpose of computing the Floer homology, it suffices to consider the $\tau$ function up until the point where it is increasing. So we only need the restricted $\tau$ function $\tau \colon [0,N+1] \to \mathbb{Z}$, or equivalently, it suffices to determine the $\Delta$-function $\Delta \colon [0,N] \to \mathbb{Z}$. Furthermore, the function $\Delta$ is completely determined on $[0,N]$ as follows \cite[Theorem 1.3]{caka}. Let
\[
G = \{ g \in \mathbb{Z}_{\ge 0} \; | \; g = bc i + ac j + ab k \text{ for some } i,j,k \in \mathbb{Z}_{\ge 0} \}
\]
be the additive semigroup generated by $bc, ac, ab$. We have that $\Delta(n) \in \{ -1 , 0 , 1 \}$ for all $n \in [0,N]$. Moreover, for any $n \in [0,N]$, we have that 
\[
\Delta(n) = \begin{cases} \; \; \, 1 & n \in G, \\ -1 & N-n \in G, \\ \; \; \, 0 & \text{otherwise}. \end{cases}
\]
It follow that for $n \in [0,N]$, we have
\begin{equation}\label{equ:taudiff}
\tau(n+1) = | G \cap [0,n] | - | G \cap [N-n,N] |
\end{equation}
where $|S|$ denotes the cardinality of a finite set $S$. From this expression it is clear that $\tau$ satisfies $\tau(N+1-n) = \tau(n)$ for all $n \in [0,N]$.

\begin{remark}\label{rem:crt}
Note that since $0 < N < abc$, it follows from the Chinese remainder theorem that any $n \in G \cap [0,N]$ can be written as $abc( i/a + j/b + k/c)$ for {\em uniquely determined} integers $i,j,k$ satisfying $0 < i < a$, $0 < j < b$, $0 < k < c$.
\end{remark}

We briefly explain how to obtain $HF^+(-\Sigma(a,b,c))$ from the $\tau$ function. First we recall the definition of a graded root:
\begin{definition}[\cite{nem}]
Let $R$ be an infinite tree with vertices $\mathcal{V}$ and edges $\mathcal{E}$. Denote by $[u,v]$ the edge joining $u$ and $v$. Edges are unordered, so $[u,v] = [v,u]$. We say that $R$ is a {\em graded root} with grading $\chi \colon \mathcal{V} \to \mathbb{Z}$ if
\begin{itemize}
\item[(a)]{$\chi(u) - \chi(v) = \pm 1$ for all $[u,v] \in \mathcal{E}$,}
\item[(b)]{$\chi(u) > \min \{ \chi(v) , \chi(w) \}$ if $[u,v], [u,w] \in \mathcal{E}$ and $v \neq w$,}
\item[(c)]{$\chi$ is bounded from below, $\chi^{-1}(k)$ is finite for any $k \in \mathbb{Z}$ and $|\chi^{-1}(k)| = 1$ for all sufficiently large $k$.}
\end{itemize}

\end{definition}

Define a partial relation $\succeq$ on $(R , \chi)$ by declaring $v \succeq w$ if there exists a sequence of vertices $w = w_0 , w_1 , \dots , w_k = v$ such that $[w_{i-1} , w_i] \in \mathcal{E}$ and $\chi(w_i) - \chi(w_{i-1}) = 1$ for $i = 1, \dots , k$. From the axioms (a)-(c), it is easily seen that for each $u \in \mathcal{V}$ there is a unique vertex $v \in \mathcal{V}$ with $[u,v] \in \mathcal{E}$ and $\chi(v) = \chi(u)+1$. From this and axiom (c), it is seen that for any two vertices $u,v \in \mathcal{V}$ there is a unique $\succeq$-minimal element $w \in \mathcal{V}$ with $w \succeq v$ and $w \succeq u$. We denote this element by $\sup(u,v)$.

For $v \in \mathcal{V}$, let $\delta_v$ be the number of edges which have $v$ as an endpoint. Clearly $\delta_v \ge 1$ and $\delta_v = 1$ if and only if $v$ is $\succeq$-minimal.

Given a graded root $(R , \chi)$ define $\mathbb{H}(R,\chi)$ to be the graded $\mathbb{F}[U]$-module constructed as follows. A basis for $\mathbb{H}(R,\chi)$ is given by $\{ e_v \; | \; v \in \mathcal{V} \}$ where $deg(e_v) = 2 \chi(v)$ and $U e_v = e_u$, where $u$ is the unique $u \in \mathcal{V}$ such that $[u,v] \in \mathcal{E}$ and $\chi(u) = \chi(v) + 1$. Note that our definition of $\mathbb{H}(R,\chi)$ differs slightly from the one in \cite{nem}. This is because we are using Floer cohomology rather than homology. The underlying graded abelian group of $\mathbb{H}(R,\chi)$ is the same as that in \cite{nem}, but our $\mathbb{F}[U]$-module action is the transpose of the one in \cite{nem}.

For a positive integer $n$ let $\mathcal{T}^+(n)$ denote the $\mathbb{F}[U]$-module $\mathbb{F}[U]/(U^n)$. For any integer $d$ and any $\mathbb{F}[U]$-module $M$, let $M_d$ be the degree-shift of $M$ defined by $(M_d)^j = M^{j-d}$. In particular, for any positive integer $n$ and any integer $d$, we have the $\mathbb{F}[U]$-module $\mathcal{T}^+_d(n) = (\mathcal{T}^+(n))_d$. We refer to $\mathcal{T}^+_d(n)$ as a {\em tower of length} $n$. The lowest degree in $\mathcal{T}^+_d(n)$ is $d$ and the highest degree is $d+2n-2$.

\begin{proposition}[\cite{nem}, Proposition 3.5.2]\label{prop:hfiso}
Let $(R , \chi)$ be a graded root. Set $I = \{ v \in \mathcal{V} \; | \; \delta_v = 1\}$. We choose an ordering on the set $I$ as follows. The first element $v_1 \in I$ is chosen so that $\chi(v_1)$ is the minimal value of $\chi$. If $v_1, \dots , v_k \in I$ are already determined and $J = \{v_1 , \dots , v_k\} \neq I$, then $v_{k+1}$ is chosen from $I \setminus J$ such that $\chi(v_{k+1}) = \min_{v \in I \setminus J} \chi(v)$. Let $w_{k+1} \in \mathcal{V}$ be the unique $\succeq$-minimal vertex of $R$ which dominates both $v_{k+1}$ and some element of $J$. Then we have an isomorphism of $\mathbb{F}[U]$-modules:
\[
\mathbb{H}(R , \chi) \cong \mathbb{F}[U]_{2\chi(v_1)} \oplus \bigoplus_{k \ge 2} \mathcal{T}^+_{2\chi(v_k)}( \chi(w_k) - \chi(v_k) ).
\]
In particular we have an isomorphism
\[
\mathbb{H}_{red}(R , \chi) \cong \bigoplus_{k \ge 2} \mathcal{T}^+_{2\chi(v_k)}( \chi(w_k) - \chi(v_k) ).
\]
\end{proposition}

Let $\tau \colon \{0 , 1, \dots , l\} \to \mathbb{Z}$ be any function and suppose there is an $l$ such that $\tau(i+1) \ge \tau(i)$ for all $i \ge l$. Following \cite[Example 3.4]{nem}, we construct a graded root $(R,\chi)$ from $\tau$ as follows. Start with the vertex set $\widetilde{\mathcal{V}} = \{ v_i^k \}$ where $0 \le i \le l$ and $k \ge \tau(i)$, edge set $\widetilde{\mathcal{E}} = [v_i^k , v_i^{k+1}]$ and define $\widetilde{\chi}$ by $\widetilde{\chi}(v_i^k) = k$. Now define $\mathcal{V} = \widetilde{\mathcal{V}}/\! \! \sim$ and $\mathcal{E} = \widetilde{\mathcal{E}}/\! \! \sim$ where for $i \le j$ we set $v_i^k \sim v_j^k$ and $[v_i^k , v_i^{k+1}] \sim [v_j^k , v_j^{k+1}]$ if $k \ge \max_{i \le u \le j} \tau(u)$. We define $\chi$ to be the function on $\mathcal{V}$ induced by $\widetilde{\chi}$. It is easily checked that $(\mathcal{V} , \mathcal{E} , \chi)$ is a graded root. We call it the {\em graded root associated to $\tau$}.


Let $1 < a < b < c$ be pairwise coprime integers and let $\tau$ be the $\tau$-function defined as in Equations (\ref{equ:tau1})-(\ref{equ:tau2}). Let $(R_\tau , \chi_\tau)$ be the associated graded root. Then from \cite{nem} it follows that up to a grading shift $HF^+(-\Sigma(a,b,c))$ is isomorphic to $\mathbb{H}( R_\tau , \chi_\tau)$:
\[
HF^+(-\Sigma(a,b,c)) \cong \mathbb{H}( R_\tau , \chi_\tau)[ u ].
\]
The precise value of the grading shift $u$ can be computed \cite[Proposition 4.7]{nem}, but we will not need this.

Observe that $\tau(0) = \tau(N+1) = 0$ and that $\tau(1) = \tau(N) = 1$. We will see in Proposition \ref{prop:taumax} that $\tau(n) \le 1$ for all $n \in [0,N+1]$, hence $1$ and $N$ are global maxima of $\tau$. Let $n \in [0,N+1]$ be a global maximum of $\tau$, that is, $\tau(n) = 1$. We will say $n$ is a {\em trivial maximum} if we either have that $\tau(i) = 1$ for all $1 \le i \le n$ or $\tau(i) = 1$ for all $n \le i \le N$.

\begin{proposition}\label{prop:topdegtau}
Suppose that all global maxima of $\tau$ are trivial. Let $\ell^+(-\Sigma(a,b,c))$ denote the highest non-zero degree in $HF^+_{red}(-\Sigma(a,b,c))$. Any non-zero element in the image of $U \colon HF^+_{red}(-\Sigma(a,b,c)) \to HF_{red}^+(-\Sigma(a,b,c))$ has degree strictly less than $\ell^+(-\Sigma(a,b,c))$.
\end{proposition}
\begin{proof}
Let $(R_\tau , \chi_\tau)$ be the graded root associated to $\tau$. From the definition of $(R_\tau , \chi_\tau)$, it is clear that $(R_\tau , \chi_\tau)$ depends only on the values of the local minima and local maxima of $\tau$. More precisely, let $S^{min}$ be the set of $m \in [0,N+1]$ such that $m=0$, $m = N+1$, or $\tau(m-1) > \tau(m)$ and there exists a $k \ge 0$ such that $\tau(i) = \tau(m)$ for $m \le i \le m+k$ and $\tau(m+k+1) > \tau(m)$. Similarly, let $S^{max}$ be the set of $M \in [1,N]$ such that $\tau(M) > \tau(M-1)$ and there exists a $k \ge 0$ such that $\tau(i) = \tau(M)$ for $M \le i \le M+k$ and $\tau(M+k+1) < \tau(M)$. Write the elements of $S^{min}$ and $S^{max}$ in increasing order as $S^{min} = \{ m_0 , m_1 , \dots , m_r \}$, $S^{max} = \{ M_1 , M_2 , \dots , M_r \}$ where $0 = m_0 < m_1 < \cdots < m_r = N+1$ and $1 = M_1 < M_2 < \cdots < M_r$. The local minima and maxima must occur in alternating order, so we have
\[
0 = m_ 0 < M_1 < m_1 < M_2 < \cdots < m_{r-1} < M_r < m_r = N+1.
\]
Furthermore, $M_1$ and $M_r$ correspond to the trivial global maxima. By assumption these are the only global maxima. Hence $\tau(M_j) \le 0$ for $1 < j < r$.

Recall that $(R_\tau , \chi_\tau)$ is constructed as follows. Start with $\widetilde{\mathcal{V}} = \{ v_i^k \}$ where $0 \le i \le N+1$ and $k \ge \tau(i)$, edge set $\widetilde{\mathcal{E}} = [v_i^k , v_i^{k+1}]$ and define $\widetilde{\chi}$ by $\widetilde{\chi}(v_i^k) = k$. Then $(R_\tau , \chi_\tau)$ is obtained by taking the quotient of this by the equivalence relation $\sim$, where for $i \le j$ we set $v_i^k \sim v_j^k$ and $[v_i^k , v_i^{k+1}] \sim [v_j^k , v_j^{k+1}]$ if $k \ge \max_{i \le u \le j} \tau(u)$.

It is easily seen that the minima of $\succeq$ on $(R_\tau , \chi_\tau)$ are precisely the elements $I = \{ v_{m_i}^{\tau(m_i)} \}_{0 \le i \le r}$. From Proposition \ref{prop:taumax} we have that $\tau(n) \le 1$ for all $n \in [0,N+1]$ and hence $\tau(m_i) \le 0$ for any local minimum. Hence we may choose a permutation $\sigma \colon \{ 0 , 1 , \dots , r \} \to \{ 0 , 1 , \dots , r\}$ such that $\tau( m_{\sigma(i)} ) \le \tau( m_{\sigma(i+1)})$ for $0 \le i \le r-1$ and $\sigma(r-1) = 0$, $\sigma(r) = r$. This gives an ordering of the set $I$ as in the statement of Proposition \ref{prop:hfiso}, namely $I = \{ v'_0 , v'_1 , v'_2 , \dots , v'_r \}$ where $v'_i = v_{m_{\sigma(i)}}^{\tau(m_{\sigma(i)})}$. For each $i \in \{1 , \dots , r\}$ and each $j \in \{0 , 1 , \dots , i-1\}$ one finds that $sup(v'_i , v'_j) = v_{m_\sigma(i)}^{K_{ij}}$, where $K_{ij}$ is the maximum of $\tau(M_a)$ for all $a$ such that $M_a$ lies in the interval joining $m_{\sigma(i)}$ and $m_{\sigma(j)}$. Now let $w_1 , w_2 , \dots , w_r$ be defined as in the statement of Proposition \ref{prop:hfiso}. Then it follows that $w_i = v_{m_{\sigma(i)}}^{K_i}$, where $K_i = \min_{0 \le j \le i-i} K_{ij}$. Now suppose that $i \neq r-1, r$. Then for all $j \in \{ 0 , \dots , i-1\}$ we have that $\sigma(i) \neq 0,r$ and $\sigma(j) \neq 0,r$. It follows that each maximum $M_a$ in the interval joining $m_{\sigma(i)}$ and $m_{\sigma(j)}$ is not trivial and hence $\tau(M_a) \le 0$. Hence $K_{ij} \le 0$ for all $i \le r-2$ and $j < i$. This also implies that $\chi(w_i) = K_{i} \le 0$ for all $i \le r-2$. On the other hand since $\tau(M_1) = \tau(M_{r}) = 1$, it follows that $K_{r-1} = K_r = 1$ and thus $\chi(w_{r-1}) = \chi(w_r) = 1$. 

From Proposition \ref{prop:hfiso}, we have an isomorphism
\[
\mathbb{H}_{red}(R_\tau , \chi_\tau) \cong \bigoplus_{k = 1}^{r} \mathcal{T}^+_{2\chi(v'_k)}( \chi(w_k) - \chi(v'_k) ).
\]
The highest degree in the tower $\mathcal{T}^+_{2\chi(v'_k)}( \chi(w_k) - \chi(v'_k) )$ is $2\chi(w_k) - 2$. For $k \neq r-1, r$ we have $2\chi(w_k)-2 \le -2$ whereas for $k = r-1,r$ we have $2\chi(w_k)-2 = 0$. Thus the highest non-zero degree in $\mathbb{H}_{red}(R_\tau , \chi_\tau)$ is $0$ and only the towers for $k = r-1 , r$ attain this degree. Moreover the length of the tower $\mathcal{T}^+_{2\chi(v'_k)}( \chi(w_k) - \chi(v'_k) )$ is $\chi(w_k) - \chi(v'_k)$ which for $k=r-1 , r$ equals $1$, since $v'_{r-1} = v_{\sigma(r-1)}^{\tau(m_{\sigma(r-1)})} = v_0^{\tau(0)} = v_0^0$ and $v'_r = v_{\sigma(r)}^{\tau(m_{\sigma(r)})} = v_r^{\tau(m_r)} = v_r^0$, so $\chi(v'_{r-1}) = \chi(v'_{r}) = 0$. It follows that any non-zero element in the image of $U \colon \mathbb{H}_{red}(R_\tau , \chi_\tau) \to \mathbb{H}_{red}(R_\tau , \chi_\tau)$ has degree strictly less than $0$. Now since $HF_{red}^+(-\Sigma(a,b,c))$ is isomorphic to a grading shift of $\mathbb{H}_{red}(R_\tau , \chi_\tau)$, it also follows that any non-zero element in the image of $U \colon HF^+_{red}(-\Sigma(a,b,c)) \to HF_{red}^+(-\Sigma(a,b,c))$ has degree strictly less than $\ell^+(-\Sigma(a,b,c))$.
\end{proof}

\begin{proposition}\label{prop:taumax}
Let $1 < a < b < c$ be pairwise coprime integers and assume that $(a,b,c) \neq (2,3,5)$. Then $\tau(n+1) \le 1$ for all $n \in [0,N]$. Furthermore, all maxima of $\tau$ are trivial, except in the following cases: $(a,b,c) = (2,3,6n-1)$, $n \ge 2$ or $(a,b,c) = (2,3,6n+1)$, $n \ge 1$.
\end{proposition}
\begin{proof}
For $n \in [0,N]$, let $\alpha_n$ be the largest element of $G$ less than or equal to $N-n$. Suppose that $j \in G \cap [0,n] \setminus \{0\}$. We claim that $j + \alpha_n \in G \cap [N-n,N]$. Clearly $j + \alpha_n \in G$ because $j, \alpha_n \in G$. Since $j \le n$ and $\alpha_n \le N-n$, we also have $j + \alpha_n \le N$. Lastly, since $j > 0$, we have $j + \alpha_n \ge N-n$ for if not then $\alpha_n$ is not the largest element of $G$ less than or equal to $N-n$. Therefore $j + \alpha_n \in G \cap [N-n,N]$. So we have constructed an injective map $\phi_n : G \cap [0,n] \setminus \{0\} \to G \cap [N-n,N]$, given by $\phi(j) = j + \alpha_n$. Therefore $| G \cap [0,n] | \le 1 + |G \cap [N-n,N]$. Comparing with Equation (\ref{equ:taudiff}), we have shown that $\tau(n) \le 1$.

Now suppose that $n_0 \in [1,N]$ is a maximum of $\tau$, so $\tau(n_0) = 1$. By the symmetry $\tau(N+1+i) = \tau(i)$ of the $\tau$ function, it suffices to consider only maxima $n_0$ such that $n_0 \ge (N+1)/2$. Since $\tau$ can only change by $\pm 1$, there exists an $n \ge n_0$ such that $\tau(i) = 1$ for $n_0 \le i \le n$ and $\tau(n+1) = 0$. Since $\tau(N+1) = 0$, we see that $n \le N$. Also we have $n \ge n_0 \ge (N+1)/2$.

Since $\tau(n) =1$ and $\tau(n+1) = 0$, we have that $\Delta(n) = \tau(n+1) - \tau(n) = -1$. Hence $N - n \in G$ and $n \notin G$. Since $N-n \in G$, it follows that $\alpha_n = N-n$. Consider the map $\phi \colon G \cap [0,n] \to G \cap [N-n,N]$ given by $\phi(j) = j + N-n$ (we already saw that $\phi(j) \in G \cap [N-n ,N]$ for $j > 0$ and we also have $\phi(0) = N-n \in G \cap [N-n,N]$). Since $\tau(n+1) = 0$, it follows that $\phi$ is a bijection. Thus every $g \in G \cap [N-n,N]$ can be written uniquely as $g = (N-n) + g'$ for some $g \in G \cap [0,n]$. Since $N-n \in G$, we may write it in the form
\[
N-n = abc \left( \frac{i_0}{a} + \frac{j_0}{b} + \frac{k_0}{c} \right)
\]
for some $i_0,j_0,k_0 \ge 0$. It follows that any $g \in G \cap [N-n,N]$ has the form
\begin{equation}\label{equ:g}
g = abc \left( \frac{i}{a} + \frac{j}{b} + \frac{k}{c} \right)
\end{equation}
where $i \ge i_0$, $j \ge j_0$, $k \ge k_0$. Furthermore, by Remark \ref{rem:crt}, this is a necessary condition. That is, if $g \in G \cap [N-n,N]$ is written as in (\ref{equ:g}) then we must have $i \ge i_0$, $j \ge j_0$, $k \ge k_0$.

Suppose that $i_0 > 0$. There exists an integer $k$ such that $1/a$ lies in the interval $[(k-1)/c , k/c]$. Thus
\[
\frac{1}{a} \le \frac{k}{c} \le \frac{1}{a} + \frac{1}{c}.
\]
Consider
\[
g = (N-n) + abc\left( -\frac{1}{a} + \frac{k}{c} \right) = abc\left( \frac{i_0-1}{a} + \frac{j_0}{b} + \frac{k_0+k}{c} \right).
\]
Then $g \in G$ and $g \ge N-n$ since $k/c - 1/a \ge 0$. But $g$ is not of the form $abc(i'/a + j'/b + k'/c)$ with $i' \ge i_0$, $j' \ge j_0$, $k' \ge k_0$, hence $g \notin G \cap [N-n,N]$. The only way this can happen is that $g > N$. Thus
\[
g = (N-n) + abc\left( -\frac{1}{a} + \frac{k}{c} \right) > N,
\]
which implies that $n < abc( -1/a + k/c )$. But $k/c \le 1/a + 1/c$, so $n < abc/c = ab$. Note that since $a<b<c$, $ab$ is the smallest positive element of $G$. Thus $i \notin G$ for all $1 \le i \le n$. So $\Delta(i) \le 0$ for $1 \le i \le n$. But we also have that
\[
\tau(n) = 1 = \sum_{i=0}^{n-1} \Delta(i) = 1 + \sum_{i=1}^{n-1} \Delta(i).
\]
Since $\Delta(i) \le 0$ for $1 \le i \le n$, the only way we can have equality is that $\Delta(i) = 0$ for $1 \le i \le n$ and hence $\tau(i) = 1$ for $1 \le i \le n$. This means that $n$ is a trivial maximum.

Next, suppose that $i_0 = 0$ and $j_0 > 0$. Then by a similar argument to the $i_0 > 0$ case, there exists a $k$ such that $1/b \le k/c \le 1/b + 1/c$. Consider
\[
g = (N-n) + abc\left( -\frac{1}{b} + \frac{k}{c} \right).
\]
Arguing as in the $i_0 > 0$ case, we see that $n$ is again a trivial maximum.

Now consider the case that $i_0 = j_0 = 0$. Hence $N-n = abc(k_0/c)$ for some $k_0 \ge 0$. We will assume that $k_0 > 0$ for if $k_0 = 0$, then $n = N$ is a trivial maximum. Furthermore, recall that we are assuming $n \ge (N+1)/2 > N/2$. Therefore
\begin{equation}\label{equ:ink0}
\frac{k_0}{c} < \frac{N}{2abc} = \frac{1}{2} - \frac{1}{2a} - \frac{1}{2b} - \frac{1}{2c}.
\end{equation}
Recall that each element of $G \cap [N-n,N]$ has the form given by (\ref{equ:g}) with $i \ge 0$, $j \ge 0$, $k \ge k_0$. It follows that there does not exist a solution to
\begin{equation}\label{equ:jb}
\frac{k_0}{c} \le \frac{j}{b} + \frac{k_0-1}{c} \le 1 - \frac{1}{a} - \frac{1}{b} -\frac{1}{c}
\end{equation}
with $j \ge 0$. For if such a $j$ exists, we would have
\[
N-n = abc \left( \frac{k_0}{c} \right) < abc\left( \frac{j}{b} + \frac{k_0-1}{c} \right) \le N
\]
and thus $g = abc( j/b + (k_0-1)/c )$ would be an element of $G \cap [N-n,N]$ not of the form $abc(i'/a + j'/b + k'/c)$ with $i' \ge i_0$, $j' \ge j_0$, $k' \ge k_0$. In particular, $j=1$ is not a solution. But since $b < c$, we have $1/b + (k_0-1)/c > k_0/c$. So it must be the second inequality in (\ref{equ:jb}) that is violated. That is, we must have
\[
\frac{j}{b} + \frac{k_0-1}{c} > 1 - \frac{1}{a} - \frac{1}{b} - \frac{1}{c}.
\]
Rearranging and using (\ref{equ:ink0}), we find that
\begin{equation}\label{equ:cond1}
\frac{1}{a} + \frac{3}{b} - \frac{1}{c} > 1.
\end{equation}
If this condition is not satisfied, then all maxima of $\tau$ are trivial.

If $a \ge 4$, then $b \ge 5$ and
\[
\frac{1}{a} + \frac{3}{b} - \frac{1}{c} < \frac{1}{4} + \frac{3}{5} < 1.
\]
So (\ref{equ:cond1}) implies $a < 4$, hence $a = 2$ or $3$.

If $a=3$, then since $1/3 +3/5 < 1$, (\ref{equ:cond1}) implies that $b \le 5$. So $b = 4$ and then (\ref{equ:cond1}) implies that $c > 12$.

If $a=2$, then since $1/2 + 3/6 - 1/c < 1$, we must have $b < 6$. Hence $b = 3$ or $5$.

If $a=2, b=5$, then (\ref{equ:cond1}) is satisfied for any $c > 10$.

If $a=2, b=3$, then (\ref{equ:cond1}) is satisfied for any $c > 3$. But we are excluding $(2,3,5)$ so $c > 5$.

To summarise, (\ref{equ:cond1}) is satisfied only in the following cases:
\begin{itemize}
\item[(1)]{$(a,b,c) = (3,4,c)$, $c > 12$.}
\item[(2)]{$(a,b,c) = (2,5,c)$, $c > 10$.}
\item[(3)]{$(a,b,c) = (2,3,c)$, $c > 5$.}
\end{itemize}

It remains to show in cases (1) and (2) we still have that all maxima of $\tau$ are trivial. We do this by a direct computation of the $\tau$ function.

In case (1) there are four subcases:
\begin{itemize}
\item[(1a)]{$(a,b,c) = (3,4,12k+1)$, $k \ge 1$.}
\item[(1b)]{$(a,b,c) = (3,4,12k+5)$, $k \ge 1$.}
\item[(1c)]{$(a,b,c) = (3,4,12k+7)$, $k \ge 1$.}
\item[(1d)]{$(a,b,c) = (3,4,12k+11)$, $k \ge 1$.}
\end{itemize}

Similarly in case (2) there are four subcases:
\begin{itemize}
\item[(2a)]{$(a,b,c) = (2,5,10k+1)$, $k \ge 1$.}
\item[(2b)]{$(a,b,c) = (2,5,10k+3)$, $k \ge 1$.}
\item[(2c)]{$(a,b,c) = (2,5,10k+7)$, $k \ge 1$.}
\item[(2d)]{$(a,b,c) = (2,5,10k+9)$, $k \ge 1$.}
\end{itemize}

Case (1a): $N = 60k-7$, $G$ is generated by $\{ 12 , 36k+3 , 48k+4 \}$. By the symmetry of the $\tau$ function, it suffices to only look for maxima of $\tau(n+1)$ with $n \in [0,N/2]$. We consider the intersection of $G$ with $[0 , N/2]$. Since $36k+3 > N/2$, all elements of $G \cap [0,N]$ have the form $12j$ for some $j$. Thus
\[
G \cap [0,N/2]  = \{ 12j \}_{0 \le j  \le \lfloor N/24 \rfloor}.
\]
Now we consider the intersection of $N-G$ with $[0,N]$. Elements of $N-G$ in the range $[0,N]$ have the form $N - 12u$, $N - (36k+3)-12u$ or $N - (48k+4)-12u$ for some $u \ge 0$. In the first case, we have
\[
N - 12u = 12(5k-1-u)+5.
\]
In the second case, we have
\[
N -(36k+3)-12u = 12(2k-1-u)+2
\]
and in the third case, we have
\[
N - (48+4) - 12u = 12(k-1-u)+1.
\]
Therefore
\[
(N-G) \cap [0,N] = \{ 12j + 5\}_{0 \le j  \le 5k-1} \cup \{ 12j+2\}_{0 \le j \le 2k-1} \cup \{ 12j+1 \}_{0 \le j \le k-1}.
\]

Partition $[0,N]$ into subintervals $I_j = [12j , 12j+11]$, $0 \le j \le 5k-2$ and $I_{5k-1} = [60k-12 , 60k-7]$. We are only interested in the subintervals which intersect with $[0,N/2]$, so we can assume $j \le \lfloor N/24 \rfloor$. Under this condition, we have $G \cap I_j \cap G = \{12j\}$ and
\[
(N-G) \cap I_j = \begin{cases} 12j+1, 12j+2, 12j+5 & 0 \le j \le k-1, \\ 12j+2, 12j+5 & k \le j \le 2k-1, \\ 12j+5 & 2k \le j. \end{cases}
\]
It follows easily that $\tau$ has only trivial maxima.

Case (1b): $N = 60k+13$, $G$ is generated by $12, 36k+15, 48k+20$. Arguing similarly to case (1a), we find
\[
G \cap [0,N/2]  = \{ 12j \}_{0 \le j  \le \lfloor N/24 \rfloor}
\]
and
\[
(N-G) \cap [0,N] = \{ 12j + 1\}_{0 \le j  \le 5k+1} \cup \{ 12j+10\}_{0 \le j \le 2k-1} \cup \{ 12j+5 \}_{0 \le j \le k-1}.
\]
By a similar argument, we see that $\tau$ has only trivial maxima.

Case (1c): $N = 60k+23$, $G$ is generated by $12, 36k+21, 48k+28$. We find
\[
G \cap [0,N/2]  = \{ 12j \}_{0 \le j  \le \lfloor N/24 \rfloor}
\]
and
\[
(N-G) \cap [0,N] = \{ 12j + 11\}_{0 \le j  \le 5k+1} \cup \{ 12j+2\}_{0 \le j \le 2k} \cup \{ 12j+7 \}_{0 \le j \le k-1}.
\]
We then see that $\tau$ has only trivial maxima.

Case (1d): $N = 60k+43$, $G$ is generated by $12, 36k+33, 48k+44$. We find
\[
G \cap [0,N/2]  = \{ 12j \}_{0 \le j  \le \lfloor N/24 \rfloor}
\]
and
\[
(N-G) \cap [0,N] = \{ 12j + 7\}_{0 \le j  \le 5k+3} \cup \{ 12j+10\}_{0 \le j \le 2k} \cup \{ 12j+11 \}_{0 \le j \le k-1}.
\]
We then see that $\tau$ has only trivial maxima.

Case (2a): $N = 30k-7$, $G$ is generated by $10, 20k+2, 50k+5$. We find
\[
G \cap [0,N/2]  = \{ 10j \}_{0 \le j  \le \lfloor N/20 \rfloor}
\]
and
\[
(N-G) \cap [0,N] = \{ 10j + 3\}_{0 \le j  \le 3k-1} \cup \{ 10j+1\}_{0 \le j \le k-1}
\]
We then see that $\tau$ has only trivial maxima.

Case (2b): $N = 30k-1$, $G$ is generated by $10, 20k+6, 50k+15$. We find
\[
G \cap [0,N/2]  = \{ 10j \}_{0 \le j  \le \lfloor N/20 \rfloor}
\]
and
\[
(N-G) \cap [0,N] = \{ 10j + 9\}_{0 \le j  \le 3k-1} \cup \{ 10j+3\}_{0 \le j \le k-1}
\]
We then see that $\tau$ has only trivial maxima.

Case (2c): $N = 30k+11$, $G$ is generated by $10, 20k+14, 50k+35$. We find
\[
G \cap [0,N/2]  = \{ 10j \}_{0 \le j  \le \lfloor N/20 \rfloor}
\]
and
\[
(N-G) \cap [0,N] = \{ 10j + 1\}_{0 \le j  \le 3k+1} \cup \{ 10j+7\}_{0 \le j \le k-1}
\]
We then see that $\tau$ has only trivial maxima.

Case (2d): $N = 30k+17$, $G$ is generated by $10, 20k+18, 50k+45$. We find
\[
G \cap [0,N/2]  = \{ 10j \}_{0 \le j  \le \lfloor N/20 \rfloor}
\]
and
\[
(N-G) \cap [0,N] = \{ 10j + 7\}_{0 \le j  \le 3k+1} \cup \{ 10j+9\}_{0 \le j \le k-1}
\]
We then see that $\tau$ has only trivial maxima.

\end{proof}

\begin{proposition}\label{prop:topdeg}
Let $1 < a < b < c$ be pairwise coprime integers and assume $(a,b,c) \neq (2,3,5)$. Let $Y = -\Sigma(a,b,c)$ and let $\ell^+(Y)$ denote the highest non-zero degree in $HF^+_{red}(Y)$. Then $\ell^+(Y) = 2\delta(Y) - 2\min\{\tau\}$. Moreover, any non-zero element in the image of $U \colon HF_{red}^+(Y) \to HF_{red}^+(Y)$ has degree strictly less than $\ell^+(Y)$.
\end{proposition}
\begin{proof}
Recall that $HF^+(Y)$ is isomorphic to $\mathbb{H}(R_\tau , \chi_\tau)$ up to an overall grading shift. The lowest degree in $HF^+(-\Sigma(a,b,c))$ is $2\delta(Y)$ and the lowest degree in $\mathbb{H}(R_\tau , \chi_\tau)$ is $2\min\{ \tau\}$, hence the grading shift is $2\delta(Y) - 2\min\{ \tau \}$. The highest non-zero degree in $\mathbb{H}(R_\tau , \chi_\tau)$ is easily seen to be $0$, hence $\ell^+(Y) = 2\delta(Y) - 2\min\{ \tau \}$.

If $(a,b,c) \neq (2,3,6n \pm 1)$ for any $n$, then all global maxima of $\tau$ are trivial by Proposition \ref{prop:taumax}. Then by Proposition \ref{prop:topdegtau} we have that any non-zero element in the image of $U \colon HF_{red}^+(Y) \to HF_{red}^+(Y)$ has degree strictly less than $\ell^+(Y)$.

If $(a,b,c) = (2,3,6n \pm 1)$ for some $n$, then it is easily seen that $HF_{red}^+(Y)$ is concentrated in a single degree and hence $U$ acts trivially on $HF_{red}^+(Y)$.

\end{proof}

\subsection{Computation of $j^{(c)}(T_{a,b})$ and $\theta^{(c)}(T_{a,b})$}\label{sec:jtheta}

Following \cite{caka}, define $\kappa(a,b,c)$ to be the cardinality of $G \cap [0,N]$.

\begin{lemma}\label{lem:kappa}
Let $\tau_1(a,b,c)$ denote the number of integers $x,y,z$ with $0 < x < a$, $0 < y < b$, $0< z < c$ and $x/a + y/b + z/c < 1$. Then $\tau_1(a,b,c) = \kappa(a,b,c)$.
\end{lemma}
\begin{proof}
By Remark \ref{rem:crt}, any $x \in G$ with $x \le N$ has a unique representation $x bc + y ac + zab$ with $0 \le x < a$, $0 \le y < b$, $0 \le z <c$. Thus $\kappa(a,b,c)$ is the number of points $(x,y,z) \in \mathbb{Z}^3_{\ge 0}$ such that
\[
x bc + y ac + z ab \le abc - bc - ac - ab.
\]
From \cite[Theorem 1.3]{caka} it follows that $N \notin G$, so $\kappa(a,b,c)$ is also the number of $x,y,z \ge 0$ such that
\[
x bc + y ac + z ab < abc - bc - ac - ab.
\]
Dividing through by $abc$, this is equivalent to
\[
\frac{x}{a} + \frac{y}{b} + \frac{z}{c} < 1 - \frac{1}{a} - \frac{1}{b} - \frac{1}{c}
\]
which can be rewritten as
\[
\frac{x+1}{a} + \frac{y+1}{b} + \frac{z+1}{q} < 1.
\]
Setting $x' = x+1, y' = y+1, z' = z+1$, we see that $\kappa(a,b,c)$ is equal to $\tau_1(a,b,c)$.
\end{proof}

Recall that $\lambda$ denotes the Casson invariant. From \cite[\textsection 19]{sav}, we have $8\lambda( \Sigma(a,b,c)) = -(a-1)(b-1)(c-1) + 4\tau_1(a,b,c)$. Thus Lemma \ref{lem:kappa} gives:
\begin{equation}\label{equ:kap}
8\lambda( \Sigma(a,b,c) ) = -(a-1)(b-1)(c-1) + 4\kappa(a,b,c).
\end{equation}

From \cite{fs}, \cite{cs}, we also have that
\[
8 \lambda(\Sigma(a,b,c)) = \sum_{j=1}^{c-1} \sigma_{T_{b,c}}(j/c) = \sigma^{(c)}(T_{a,b}).
\]

\begin{theorem}\label{thm:theta}
Let $a,b,c > 1$ be coprime integers and suppose that $c$ is a prime number. Then 
\[
j^{(c)}( -T_{a,b} ) = \begin{cases} \kappa(a,b,c) & \text{if } c \text{ is odd}, \\ 2\kappa(a,b,c) & \text{if } c=2 \end{cases}
\]
and
\[
\theta^{(c)}( T_{a,b} ) = (a-1)(b-1)/2.
\]
\end{theorem}
\begin{proof}
Let $Y = -\Sigma(a,b,c)$. If $(a,b,c)$ is a permutation of $(2,3,5)$, then $HF^+_{red}(Y) = 0$ and $\kappa(2,3,5) = 0$. Then from \cite[Proposition 3.16]{bh}, it follows that $j^{(c)}(-T_{a,b}) = 0$ and $\theta^{(c)}(T_{a,b}) = -\sigma^{(c)}(T_{a,b})/(c-1)$. But since $\kappa(2,3,5) = 0$, Equation (\ref{equ:kap}) gives $\sigma^{(c)}(T_{a,b}) = 8 \lambda(\Sigma(a,b,c)) = -(a-1)(b-1)(c-1)$. Hence $\theta^{(c)}(T_{a,b}) = (a-1)(b-1)/2$.

Henceforth we assume that $(a,b,c)$ is not a permutation of $(2,3,5)$ and hence $HF^+_{red}(Y) \neq 0$. Recall that $\Sigma(a,b,c)$ is the boundary of a negative definite plumbing \cite{nr} whose plumbing graph has only one bad vertex in the sense of \cite{os2}. Then from \cite[Corollary 1.4]{os2}, we have that $HF^+(Y)$ is concentrated in even degrees. As explained in \cite[\textsection 7]{bh} we have $\Sigma(a,b,c) = \Sigma_c(T_{a,b})$ and the generator of the $\mathbb{Z}_c$-action on $\Sigma(a,b,c)$ is isotopic to the identity. Therefore in the spectral sequence $\{ E_r^{p,q} , d_r \}$ for the equivariant Seiberg--Witten--Floer cohomology of $Y$, we have $E_2^{0,q} = H^0( \mathbb{Z}_c ; HSW^q(Y)) \cong HSW^q(Y)$. Furthermore, the graded roots algorithm implies that $HF_{red}^+(Y)$ is concentrated in degrees $d(Y)$ and above. It follows that there can be no differentials in the spectral sequence and hence $\delta_{\mathbb{Z}_c , S^0}(Y) = \delta(Y)$. Recall that $\delta^{(c)}_j( K ) = -\sigma^{(c)}(K)/2$ for all sufficiently large $j$. Then since $Y = -\Sigma(a,b,c) = \Sigma_c( -T_{a,b})$, it follows that $$\delta_{\mathbb{Z}_c , S^j}(Y) = -\sigma^{(c)}(-T_{a,b})/8 = \lambda( \Sigma(a,b,c) ) = -\lambda(Y)$$ for sufficiently large $j$. But since $HF^*_{red}(Y)$ is concentrated in even degrees, we also have $dim_{\mathbb{F}}(HF^*_{red}(Y)) = \delta(Y) + \lambda(Y)$. Hence for large enough $j$, we have $\delta_{\mathbb{Z}_c , S^0}(Y) - \delta_{\mathbb{Z}_c , S^j}(Y) = \delta(Y) + \lambda(Y) = dim_{\mathbb{F}}(HF^*_{red}(Y))$. From Proposition \ref{prop:topdeg}, it follows that any element in the image of $U \colon HF_{red}^+(Y , \mathfrak{s}) \to HF_{red}^+(Y , \mathfrak{s})$ has degree strictly less than $\ell^+(Y,\mathfrak{s})$. Thus conditions (1)-(4) of Proposition \ref{prop:jinv} are met. Therefore (since $c$ is an odd prime and $Y$ is an integral homology sphere) we have
\[
j'(Y) = \ell^+(Y)/2 - \delta_{\mathbb{Z}_c , S^0}(Y) + \delta(Y) + \lambda(Y) = \ell^+(Y)/2 + \lambda(Y),
\]
where $j'(Y) = j^{(c)}(Y)$ if $c$ is odd and $j'(Y) = j^{(2)}(Y)/2$ if $c=2$. Moreover, Proposition \ref{prop:topdeg} also gives $\ell^+(Y) = 2( \delta(Y) -\min(\tau))$. From \cite[\textsection 5]{caka}, we also have that $-\min(\tau) + \delta(Y) = \kappa(a,b,c) - \lambda(Y)$. Thus $\ell^+(Y) = 2(\kappa - \lambda(Y))$. Hence
\[
j'(-T_{a,b}) = j^{(c)}(Y) = \ell^+(Y)/2 + \lambda(Y) = \kappa(a,b,c),
\]
which gives $j^{(c)}(-T_{a,b}) = \kappa(a,b,c)$ if $c$ is odd and $j^{(2)}(-T_{a,b}) = 2 \kappa(a,b,c)$ if $c=2$. We also have $\sigma^{(c)}( T_{a,b} ) = -8 \lambda(Y) = 8 \lambda( \Sigma(a,b,c) )$ and thus we have 
\begin{align*}
\theta^{(c)}(T_{a,b}) &= \max\left\{ 0 , \frac{2j'(Y)}{(c-1)} - \frac{\sigma^{(c)}(T_{a,b})}{2(c-1)} \right\} \\
&= \max\left\{ 0 , \frac{ 2\kappa(a,b,c) -4\lambda( \Sigma(a,b,c) )}{(c-1)} \right\} \\
&= \frac{1}{2}(a-1)(b-1)
\end{align*}
where the last line follows from Equation (\ref{equ:kap}).
\end{proof}

\section{Branched covers}\label{sec:branched}

Let $Y = \Sigma( a'_1 , a'_2 , \dots , a'_r)$ be a Brieskorn homology sphere and $p$ a prime such that $p$ divides $a'_1 a'_2 \cdots a'_r$. Without loss of generality we may assume $p$ divides $a'_1$. We set $a_1 = a'_1/p$ and $a_j = a'_j$ for $j > 1$. So $Y = \Sigma(pa_1 , a_2 , \dots , a_r)$. Then $\mathbb{Z}_p$ acts on $Y$ with quotient space $Y_0 = Y/\mathbb{Z}_p = \Sigma(a_1, \dots , a_r)$ and the quotient map $Y \to Y_0$ is a $p$-fold cyclic branched cover.

\begin{theorem}\label{thm:branch}
We have that 
\[
\delta(-Y) - \delta^{(p)}_\infty(-Y) \ge {\rm rk}( HF_{red}^+(Y)) - p \, {\rm rk}( HF_{red}^+(Y_0) ).
\]
\end{theorem}
\begin{proof}
Let $W_0$ be the negative definite plumbing bounded by $Y_0$. Since $Y_0$ is an integral homology sphere it follows that $H_1(W_0 ; \mathbb{Z}) = 0$. Let $k \subset Y_0$ denote the branch locus of $Y \to Y_0$. Then $k$ is a knot in $Y_0$. Let $\Sigma \subset W_0$ be the pushoff of a Seifert surface for $k$, so $\Sigma$ is a properly embedded surface in $W_0$ which meets $\partial W_0 = Y_0$ in $k$. Let $W \to W_0$ be the $p$-fold cyclic cover of $W_0$ branched over $\Sigma$. Then $W$ has boundary $Y$ and the $\mathbb{Z}_p$-action on $Y$ extends to $W$. By \cite[Proposition 2.5]{bar}, for any characteristic $c \in H^2(W_0 ; \mathbb{Z})$ there exists a $\mathbb{Z}_p$-invariant spin$^c$-structure $\mathfrak{s}$ on $W$ such that $c_1(\mathfrak{s}) = \pi^*(c)$ in $H^2(W ; \mathbb{Q})$. Now we apply the equivariant Fr{\o}yshov inequality \cite[Theorem 5.3]{bh} to $W$ giving $\delta_\infty(-Y) + \delta(W , \mathfrak{s}) \le 0$, where
\begin{align*}
\delta(W , \mathfrak{s}) &= \frac{ c_1(\mathfrak{s})^2 - \sigma(W) }{8} \\
&= \frac{ p c^2 - \sigma(W) }{8} \\
&= p \left( \frac{ c^2 - \sigma(W_0) }{8} \right) + \frac{ p\sigma(W_0) - \sigma(W)}{8}.
\end{align*}
The maximum of $(c^2  -\sigma(W_0))/8$ over all characteristics of $H^2( W_0 ; \mathbb{Z})$ equals $\delta(Y_0)$ \cite[Theorem 8.3]{nem}. Hence we obtain
\[
\delta^{(p)}_\infty(-Y) + p \delta(Y_0) + \frac{ p\sigma(W_0) - \sigma(W)}{8} \le 0,
\]
which we may rewrite as
\[
\delta(-Y) - \delta^{(p)}_\infty(-Y) \ge -(\delta(Y) - p \delta(Y_0) ) + \frac{ p\sigma(W_0) - \sigma(W)}{8}.
\]

Next, we claim that $(p \sigma(W_0) - \sigma(W) )/8 = p \lambda(Y_0) - \lambda(Y)$, where $\lambda(Y_0), \lambda(Y)$ are the Casson invariants of $Y_0$ and $Y$. Assuming this claim for the moment, our inquality becomes
\[
\delta(-Y) - \delta^{(p)}_\infty(-Y) \ge -(\delta(Y)+\lambda(Y)) + p( \delta(Y_0) + \lambda(Y_0) ).
\]
But since $HF_{red}^+(Y)$ and $HF_{red}^+(Y_0)$ are concentrated in odd degrees, \cite[Theorem 1.3]{os} gives
\[
{\rm rk}( HF_{red}^+(Y)) = -\delta(Y) - \lambda(Y), \quad {\rm rk}( HF_{red}^+(Y_0)) = -\delta(Y_0) - \lambda(Y_0),
\]
and hence we obtain
\[
\delta(-Y) - \delta^{(p)}_\infty(-Y) \ge {\rm rk}( HF_{red}(-Y)) - p \, {\rm rk}( HF_{red}(-Y_0) ).
\]

It remains to prove the claim that $(p \sigma(W_0) - \sigma(W) )/8 = p \lambda(Y_0) - \lambda(Y)$. From \cite[Theorem 2]{cosa}, we have that
\[
\lambda^{\mathbb{Z}/p}(Y) - p \lambda(Y_0) = \frac{\sigma(W) - p \sigma(W_0)}{8}
\]
where $\lambda^{\mathbb{Z}/p}(Y)$ is the equivariant Casson invariant of $Y$ with respect to the $\mathbb{Z}_p$-action \cite{cosa}. Furthermore \cite[Theorem 3]{cosa} implies that $\lambda^{\mathbb{Z}/p}(Y) = \lambda(Y)$, because $k \subset Y_0 = \Sigma(a_1 , \dots , a_r)$ is a fibre of the Seifert fibration on $Y_0$, so it is a graph knot in the terminology of \cite[\textsection 5]{cosa}. This proves the claim that $\lambda(Y) - p \lambda(Y_0) = (\sigma(W) - p \sigma(W_0))/8$.
\end{proof}

Let $Y,Y_0$ be as in Theorem \ref{thm:branch}. Then by \cite[Theorem 1.1]{kali}, we have an inequality ${\rm rk}( HF_{red}^+(Y)) \ge p \, {\rm rk}( HF_{red}^+(Y_0) )$. Combined with Theorem \ref{thm:branch}, this gives:
\[
p \, {\rm rk}( HF_{red}^+(Y_0) ) \le {\rm rk}( HF_{red}^+(Y))  \le p \, {\rm rk}( HF_{red}^+(Y_0) ) + (\delta(-Y) - \delta^{(p)}_\infty(-Y) ).
\]
In particular, the equality $\delta^{(p)}_\infty(-Y) = \delta(-Y)$ can only happen if ${\rm rk}( HF_{red}^+(Y)) = p \, {\rm rk}( HF_{red}^+(Y_0) )$. Hence we obtain:

\begin{corollary}
Let $Y,Y_0$ be as in Theorem \ref{thm:branch}. If ${\rm rk}( HF_{red}^+(Y)) > p \, {\rm rk}( HF_{red}^+(Y_0) )$ then $\delta^{(p)}_\infty(-Y) < \delta(-Y)$.
\end{corollary}

\section{Free actions}\label{sec:free}

Suppose $p$ does not divide $a_1 a_2 \cdots a_r$. Then the restriction of the $S^1$-action on $Y = \Sigma(a_1 , \dots , a_r)$ acts freely on $Y$. Let $Y_0 = Y/\mathbb{Z}_p$ be the quotient. Then $Y$ is a rational homology sphere. In fact if $Y = M( e_0 , (a_1 , b_1) , \dots , (a_r , b_r))$, then it is easily seen that $Y_0 = M( pe_0 , (a_1 , pb_1) , \dots , (a_r , pb_r) )$. It is easy to see that $H_1(Y_0 ; \mathbb{Z}) \cong \mathbb{Z}_p$ and thus $Y_0$ has exactly $p$ spin$^c$-structures. Further, the pullback to $Y$ of any spin$^c$-structure on $Y_0$ must coincide with the unique spin$^c$-structure on $Y$.

\begin{theorem}\label{thm:free}
For any spin$^c$-structure $\mathfrak{s}_0$ on $Y_0$, we have 
\[
\delta^{(p)}_\infty(Y) - \delta(Y) = {\rm rk}( HF_{red}^+(Y) ) - {\rm rk}( HF_{red}^+(Y_0 , \mathfrak{s}_0 ) ).
\]
\end{theorem}
\begin{proof}
Set $G = \mathbb{Z}_p$ and $H^*_G = H^*_G( pt ; \mathbb{F})$. By choosing a $G$-invariant metric on $Y$, we may construct a $S^1 \times G$-equivariant Conley index for $(Y,\mathfrak{s})$, which we denote by $I(Y , \mathfrak{s})$, here $\mathfrak{s}$ denotes the unique spin$^c$-structure on $Y$. As shown in \cite[\textsection 3]{lima}, the Conley index $I(Y_0 , \mathfrak{s}_0)$ for $(Y_0 , \mathfrak{s}_0)$ can be identified with the $\mathbb{Z}_p$-fixed point set $I(Y , \mathfrak{s})^{\mathbb{Z}_p}$ of $I(Y , \mathfrak{s})$. Thus we have isomorphisms
\begin{align*}
\widetilde{H}^*_{S^1 \times G}( I(Y , \mathfrak{s})^{\mathbb{Z}_p} ; \mathbb{F} ) & \cong \widetilde{H}^*_{S^1 \times G}( I(Y_0 , \mathfrak{s}_0) ; \mathbb{F}) \\
& \cong \widetilde{H}^*_{S^1}( I(Y_0 , \mathfrak{s}_0) , \mathbb{F}) \otimes_{\mathbb{F}} H^*_G \\
& \cong HSW^*( Y_0 , \mathfrak{s}_0) \otimes_{\mathbb{F}} H^*_G.
\end{align*}
It should be noted that the above isomorphisms only preserve relative gradings. Recall that $H^*_G \cong \mathbb{F}[Q]$, $deg(Q) = 1$ if $p=2$ and $H^*_G \cong \mathbb{F}[R,S]/(R^2)$, $deg(R) = 1$, $deg(S) = 2$ if $p$ is odd. In the case $p=2$, define $S = Q^2$. Let $\mathcal{S} = \{ 1 , S , S^2 , \dots \}$. Then $\mathcal{S}$ is a multiplicative subset of $H^*_{S^1 \times G} = H^*_{S^1 \times G}(pt ; \mathbb{F}) \cong H^*_G[U]$. The localisation theorem in equivariant cohomology \cite[III, Theorem 3.8]{die} applied to the pair $( I(Y , \mathfrak{s} ) , I(Y , \mathfrak{s})^{\mathbb{Z}_p})$ and multiplicative set $\mathcal{S}$ implies that the inclusion $I(Y , \mathfrak{s})^{\mathbb{Z}_p} \to I(Y , \mathfrak{s})$ induces an isomorphism
\begin{equation}\label{equ:lociso}
S^{-1} \widetilde{H}^*_{S^1 \times G}( I(Y , \mathfrak{s}) ; \mathbb{F} ) \to S^{-1} \widetilde{H}^*_{S^1 \times G}( I(Y , \mathfrak{s})^{\mathbb{Z}_p} ; \mathbb{F} ).
\end{equation}
Since $\widetilde{H}^*_{S^1 \times G}( I(Y , \mathfrak{s}) ; \mathbb{F} ) \cong HSW^*_{\mathbb{Z}_p}$ and as we have shown above, $\widetilde{H}^*_{S^1 \times G}( I(Y , \mathfrak{s})^{\mathbb{Z}_p} ; \mathbb{F} ) \cong HSW^*( Y_0 , \mathfrak{s}_0) \otimes_{\mathbb{F}} H^*_G$, we get an isomorphism
\[
S^{-1} HSW^*_{\mathbb{Z}_p}( Y , \mathfrak{s} ) \cong HSW^*(Y_0 , \mathfrak{s}_0) \otimes_{\mathbb{F}} S^{-1} H^*_G.
\]
This is an isomorphism of relatively graded $H^*_{S^1 \times G}$-modules.

For the rest of the proof we restrict to the case that $p$ is odd. The proof in the case that $p=2$ is similar. We have
\[
HSW^*(Y_0 , \mathfrak{s}_0) \cong HF^+_*(Y_0 , \mathfrak{s}_0) \cong \mathbb{F}[U] \theta_0 \oplus HF_{red}^+(Y_0 , \mathfrak{s}_0)
\]
for some $\theta_0$. Combined with (\ref{equ:lociso}), we have an isomorphism
\begin{equation}\label{equ:lociso1}
S^{-1} HSW^*_{\mathbb{Z}_p}( Y , \mathfrak{s} ) \cong \frac{\mathbb{F}[U,R,S,S^{-1}]}{(R^2)} \theta_0 \oplus \frac{HF_{red}^+(Y_0, \mathfrak{s}_0)[R,S,S^{-1}]}{(R^2)}.
\end{equation}
On the other hand, the proof of Proposition \ref{prop:deltaprop} gives an isomorphism
\[
HSW^*_{\mathbb{Z}_p}(Y , \mathfrak{s}) \cong E_\infty(Y) \cong \frac{\mathbb{F}[U,R,S]}{(R^2)} U^m \theta \oplus E_\infty(Y)_{red}
\]
under which the $\mathbb{F}[U]$-module structure is given by an endomorphism of the form $\widehat{U} = U_{(0,2)} + U_{(1,1)} + \cdots $ with $U_{(0,2)} = U$. Localising with respect to $S$ gives an isomorphism
\[
S^{-1} HSW^*_{\mathbb{Z}_p}(Y , \mathfrak{s}) \cong E_\infty(Y) \cong \frac{\mathbb{F}[U,R,S,S^{-1}]}{(R^2)} U^m \theta \oplus S^{-1} E_\infty(Y)_{red}.
\]
Define $W \subseteq S^{-1} HSW^*_{\mathbb{Z}_p}( Y , \mathfrak{s} ) )$ to be the set of $x \in S^{-1} HSW^*_{\mathbb{Z}_p}( Y , \mathfrak{s} ) )$ such that for each $j \ge 0$, there exists a $k \ge 0$ for which $U^k x \in \mathcal{F}_j(Y)$. The proof of Proposition \ref{prop:deltaprop} demonstrates that $W \cong S^{-1} E_\infty(Y)_{red}$. On the other hand, the isomorphism (\ref{equ:lociso1}) clearly shows that $W \cong HF_{red}^+(Y_0, \mathfrak{s}_0)[R,S,S^{-1}]/(R^2)$. Combining these, we have an isomorphism
\[
S^{-1} E_\infty(Y)_{red} \cong HF_{red}^+(Y_0, \mathfrak{s}_0)[R,S,S^{-1}]/(R^2).
\]

In any fixed degree $j$, the rank of $(HF_{red}^+(Y_0, \mathfrak{s}_0)[R,S,S^{-1}]/(R^2))^j$ is equal to ${\rm rk}( HF_{red}^+(Y_0 , \mathfrak{s}_0))$, hence the same is true of $S^{-1} E_\infty(Y)_{red}$. From \cite[Lemma 5.7]{bh}, we have that $S \colon E_{r}^{p,q}(Y) \to E_r^{p+2,q}(Y)$ is an isomorphism for all large enough $p$. Hence for large enough $p$, $E_r^{2p,*}(Y)$ is independent of $p$ and we denote the resulting group by $M_r^*$. We similarly define $M_\infty^*$. Clearly the rank of $(S^{-1} E_\infty(Y)_{red})^j$ for any $j$ equals ${\rm rk}( M_\infty )_{red}$. So we have proven that
\[
{\rm rk}(HF_{red}^+(Y_0 , \mathfrak{s}_0) ) = {\rm rk}(M_\infty)_{red}.
\]
Set $s_r = {\rm rk}(M_r)_{red}$. It follows from \cite[Lemma 5.8]{bh} that $(M_{r+1})_{red}$ is a subquotient of $(M_r)_{red}$. Hence the sequence $s_2 , s_3, \dots $ is decreasing and equals ${\rm rk}(M_\infty)_{red}$ for sufficiently large $r$. Furthermore, $M_2 \cong HF^+(Y)$, so $s_2 = {\rm rk}( HF_{red}^+(Y) )$.

Recall from the proof of Proposition \ref{prop:deltaprop} that 
\[
E_r(Y) \cong \frac{\mathbb{F}[U,R,S]}{(R^2)} U^{m_r}\theta \oplus E_r(Y)_{red}
\]
for some increasing sequence $0 = m_0 \le m_1 \le \dots$. Suppose $m_{r+1} > m_r$. Then $d_r( U^{m_r + j} \theta ) \neq 0$ for $0 \le j \le m_{r+1} - m_r$. Notice that $U^{m_r+j}\theta$ has bi-degree $(0 , a)$ where $a = 2m_r + 2j + d(Y)$ is even, hence $d_r( U^{m_r + j} \theta )$ has bi-degree $( r , a + (1-r) )$. But all elements of $E_r(Y)_{red}$ have bi-degree $(u,v)$ with $v$ odd (because $E_r(Y)_{red}$ is a subquotient of $HF_{red}^+(Y)[ R , S ]/(R^2)$ and $HF_{red}^+(Y)$ is concentrated in odd degrees). Hence $a + (1-r)$ must be odd, which means that $r$ is even. So we can regard $d_r$ as mapping into $(M_r)_{red}$. Hence the rank of $(M_{r+1})_{red}$ is at least $m_{r+1} - m_{r}$ less than the rank of $(M_r)_{red}$, that is,
\begin{equation}\label{equ:srmr}
s_{r} - s_{r+1} \ge m_{r+1} - m_r.
\end{equation}
We claim that this inequality is actually an equality. Equivalently $M_{r+1}$ is the quotient of $M_r$ by the span of $\{ d_r( U^{m_r + j} ) \}$, $0 \le j \le m_{r+1} - m_r$. This is also equivalent to saying that $d_r( x ) = 0$ for all $r \ge 2$ and all $x \in E_r(Y)_{red}$. Consider a non-zero homogeneous element $x \in E_r(Y)_{red}^{a,b}$. Since $E_r(Y)_{red}$ is a subquotient of $HF_{red}^+(Y)[ R , S ]/(R^2)$ and $HF_{red}^+(Y)$ is concentrated in odd degrees, we have that $b$ is odd. Then $d_r(x)$ has bi-degree $(a+r , b+1-r)$. If $d_r(x) \neq 0$, then $b+1-r$ must be odd and so $r$ must be odd. Thus in order to prove the claim, it is sufficient to show that $d_r = 0$ for all odd $r$.

Consider the equivariant Seiberg--Witten--Floer cohomology of $Y$ with integral coefficients $HSW^*_{\mathbb{Z}_p}( Y ; \mathbb{Z})$. This is a module over $H^*_{\mathbb{Z}_p}( pt ; \mathbb{Z} ) \cong \mathbb{Z}[ S ]/( pS )$ where $deg(S) = 2$. The key point to observe here is that $H^*_{\mathbb{Z}_p}( pt ; \mathbb{Z} )$ is concentrated in even degrees. There is a spectral sequence $(E_r^{p,q}(Y ; \mathbb{Z}) , d_r)$ and a filtration $\{ \mathcal{F}_j(Y ; \mathbb{Z}) \}$ such that $E_\infty(Y ; \mathbb{Z})$ is the associated graded module of the filtration. The mod $p$ reduction map $\mathbb{Z} \to \mathbb{Z}_p = \mathbb{F}$ induces a morphism of $HSW^*( Y ; \mathbb{Z}) \to HSW^*(Y ; \mathbb{F})$. There is also a reduction map of the equivariant Floer cohomologies, the filtrations and the spectral sequences. The results of \cite{os2} also hold with integer coefficients so
\[
HF^+(Y ; \mathbb{Z}) \cong \mathbb{Z}[U]_{d(Y)} \oplus HF_{red}^+(Y ; \mathbb{Z})
\]
and $HF_{red}^+(Y ; \mathbb{Z})$ is concentrated in odd degrees. By the universal coefficient theorem $HF_{red}^+(Y ; \mathbb{F}) \cong HF_{red}^+(Y ; \mathbb{Z}) \otimes_{\mathbb{Z}} \mathbb{F}$. Furthermore, we have that
\begin{align*}
E_2(Y ; \mathbb{Z}) &\cong H^*( \mathbb{Z}_p ; HSW^*(Y ; \mathbb{Z} )) \\
& \cong \frac{\mathbb{Z}[U , S]}{(pS)} \theta \oplus \frac{HF_{red}^+(Y ; \mathbb{Z})[S]}{(pS)}.
\end{align*}
Clearly every homogeneous element in $\mathbb{Z}[U , S]/(pS) \theta$ has bi-degree $(a,b)$ where $a$ and $b$ are even and every homogeneous element in $HF_{red}^+(Y ; \mathbb{Z})[S]/(pS)$ has bi-degree $(a,b)$ where $a$ is even and $b$ is odd. From this and the fact that the image of $d_r \colon E_r(Y ; \mathbb{Z} ) \to E_r(Y ; \mathbb{Z})$ is contained in $E_r(Y ; \mathbb{Z})_{red}$ it follows easily that $d_r \colon E_r(Y ; \mathbb{Z} ) \to E_r( Y ; \mathbb{Z})$ is zero for odd $r$.

Now using induction on $r$ one shows the following properties hold: (1) any $x \in E_r^{a,b}(Y)$ with $a$ even is in the image of the reduction map $E_r^{a,b}(Y ; \mathbb{Z}) \to E_r^{a,b}(Y)$, (2) every $x \in E_r^{a,b}(Y)$ with $a$ odd is of the form $x = Ry$ where $y$ is in the image of the reduction map $E_r^{a,b}(Y ; \mathbb{Z}) \to E_r^{a,b}(Y)$ and (3) $d_r \colon E_r(Y) \to E_r(Y)$ is zero for even $r$. In particular, this proves the claim that $d_r \colon E_r(Y) \to E_r(Y)$ is zero for all odd $r$ and hence the inequality (\ref{equ:srmr}) is actually an equality:
\begin{equation}\label{equ:srmr2}
s_{r} - s_{r+1} = m_{r+1} - m_r.
\end{equation}
Let $s = {\rm rk}( (M_\infty)_{red} ) = {\rm rk}( HF_{red}^+(Y_0 , \mathfrak{s}_0) )$ and let $m = \lim_{r \to \infty} m_r$. Summing (\ref{equ:srmr2}) from $r=2$ to infinity gives $s_2 - s = m - m_2$. But $m_2 = 0$ and $s_2 = {\rm rk}( HF_{red}^+(Y))$, so we get $m = {\rm rk}( HF_{red}^+(Y)) - {\rm rk}( HF_{red}^+(Y_0 , \mathfrak{s}_0))$. But recall from the proof of Proposition \ref{prop:deltaprop} that $\delta_\infty^{(p)}(Y) = \delta(Y) + m$, hence we get
\[
\delta_\infty^{(p)}(Y) = \delta(Y) + {\rm rk}( HF_{red}^+(Y)) - {\rm rk}( HF_{red}^+(Y_0 , \mathfrak{s}_0)).
\]
\end{proof}

We will take $\mathfrak{s}_0$ to be the restriction to $Y_0$ of the canonical spin$^c$-structure on the negative definite plumbing which $Y_0$ bounds. With this choice of spin$^c$-structure, the computation of $HF^+(-Y_0 , \mathfrak{s}_0)$ is easily obtained through the graded roots algorithm \cite{nem}. Let $\Delta_p(n)$ denote the delta function for $(Y_0 , \mathfrak{s}_0)$. Then (since $e_0 < 0$) we have
\begin{equation}\label{equ:deltap}
\Delta_p(n) = 1 + np e_0 - \sum_{j=1}^{r} \left\lceil \frac{npb_j}{a_j} \right\rceil
\end{equation}

\begin{proposition}\label{prop:deltap}
We have that $\Delta_p(n) \ge 0$ for $n > N/p$. In particular, if $p > N$, then $HF_{red}^+(-Y_0 , \mathfrak{s}_0 ) = 0$.
\end{proposition}
\begin{proof}
Following \cite[\textsection 3]{caka}, we let $f(x) = \lceil x \rceil - x$. Then from Equations (\ref{equ:integral}) and (\ref{equ:deltap}), we see that
\[
\Delta_p(n) = 1 + \frac{np}{a_1 \cdots a_r} - \sum_{j=1}^r f\left( \frac{ npb_j}{a_j} \right).
\]
For any integer $m$ we have $f(m/a_j) \le 1-1/a_j$, hence
\begin{equation}\label{equ:deltaest}
\Delta_p(n) \ge 1 + \frac{np}{a_1 \cdots a_r} - r + \sum_{j=1}^{r} \frac{1}{a_j}.
\end{equation}
Now suppose that $pn > N$. Thus
\[
pn > a_1 \cdots a_n \left( (r-2) - \sum_{j=1}^r \frac{1}{a_j} \right).
\]
Re-arranging, we get
\[
(1-r) + \frac{np}{a_1 \cdots a_r} + \sum_{j=1}^{r} \frac{1}{a_j} > -1.
\]
Combined with (\ref{equ:deltaest}), we get $\Delta_p(n) > -1$. But $\Delta_p$ is integer-valued, so $\Delta_p(n) \ge 0$.
\end{proof}

\begin{lemma}\label{lem:rk}
We have that ${\rm rk}( HF_{red}^+(-Y) ) > {\rm rk}( HF_{red}^+( -Y_0 , \mathfrak{s}_0 ))$ unless $r=3$ and up to reordering $(a_1,a_2,a_3)$ is one of the following:
\[
(2,3,5), (2,3,7), (2,3,11), (2,3,13), (2,3,17), (2,5,7), (2,5,9), (3,4,5), (3,4,7).
\]
\end{lemma}
\begin{proof}
We assume that $Y$ is not $\Sigma(2,3,5)$ and hence $N > 0$. Let $N_p = \lfloor N/p \rfloor$. By Proposition \ref{prop:deltap} we only need to consider $\Delta_p(x)$ for $x \in [0 , N_p]$.

In the proof we make use of abstract delta sequences \cite[\textsection 3]{kali}. Let $P_p = G \cap [ 0 , N_p]$. Consider the map $\phi \colon P_p \to [0,N]$ given by $\phi(n) = np$. From Equation (\ref{equ:deltap}) it is immediately clear that $\Delta(np) = \Delta_p(n)$, hence $\phi$ identifies the delta sequence $( P_p , \Delta_p)$ with a subsequence of the delta sequence $( P , \Delta)$. This gives the inequality ${\rm rk}( HF_{red}^+(-Y) ) \ge {\rm rk}( HF_{red}^+( -Y_0 , \mathfrak{s}_0 ))$ by \cite[Proposition 3.5]{kali}. To refine this to a strict inequality we will show that (except for the cases listed in the statement of the lemma) there exists an $x \in P \setminus \phi(P_p)$ such that $2x \le N$. In this case $\{ x , N -x \}$ defines a delta subsequence of $( P , \Delta )$ disjoint from $\phi(P_p)$ and then by \cite[Corollary 3.12]{kali} we obtain a strict inequality ${\rm rk}( HF_{red}^+(-Y) ) > {\rm rk}( HF_{red}^+( -Y_0 , \mathfrak{s}_0 ))$, as the module $\mathbb{H}_{red}$ corresponding to $\{ x , N - x \}$ has rank $1$.

We seek an $x \in P \setminus \phi(P_p)$ such that $2x \le N$. Equivalently $x \in P$, $p$ does not divide $x$ and $2x \le N$. Reorder $a_1, \dots  ,a_r$ such that $a_1 > a_2 > \cdots > a_r$. Consider $x = a_2 a_3 \cdots a_r$. It follows that $p$ does not divide $x$ as $p$ is coprime to $a_2, \dots , a_r$ and we also have that $x \in P$. Hence we obtain a strict rank inequality provided that $2x \le N$. Suppose on the contrary that $2x > N$, that is,
\[
2a_2 \cdots a_r > a_1 \cdots a_r \left( (r-2) - \sum_{j=1}^{r} \frac{1}{a_j} \right).
\]
Re-arranging, this is equivalent to
\begin{equation}\label{equ:condition}
\frac{3}{a_1} + \frac{1}{a_2} + \frac{1}{a_3} + \cdots + \frac{1}{a_r} \ge (r-2).
\end{equation}
Since $a_j \ge 2$ for all $j$ and at most one can be equal to $2$, we get that
\[
\frac{3}{a_1} + \frac{1}{a_2} + \cdots + \frac{1}{a_r} < \frac{3}{a_1} + \frac{(r-1)}{2},
\]
hence
\[
\frac{3}{a_1} + \frac{(r-1)}{2} > (r-2).
\]
From $a_1 > a_2 > \cdots > a_r \ge 2$, we see that $a_r \ge r+1$, hence
\[
\frac{3}{r+1} + \frac{(r-1)}{2} > (r-2)
\]
which simplifies to $(r+1)(r-3) < 6$. If $r \ge 5$, then $(r+1)(r-3) \ge 12$, so this leaves only the cases $r=3,4$.

If $r=4$, then since $a_4 \ge 2$, $a_3 \ge 3$, $a_2 \ge 4$, $a_1 \ge 5$, we have
\[
\frac{3}{a_1} + \frac{1}{a_2} + \frac{1}{a_3} + \frac{1}{a_4} \le \frac{3}{5} + \frac{1}{4} + \frac{1}{3} + \frac{1}{2} = \frac{101}{60} < 2.
\]
Hence (\ref{equ:condition}) can not be satisfied with $r=4$.

Now suppose that $r=3$. If $a_3 \ge 4$, then $a_2 \ge 5$, $a_1 \ge 6$ and so
\[
\frac{3}{a_1} + \frac{1}{a_2} + \frac{1}{a_3} \le \frac{3}{6} + \frac{1}{5} + \frac{1}{4} = \frac{19}{20} < 1
\]
and so (\ref{equ:condition}) can only be satisfied if $a_3 = 2$ or $3$.

If $a_3 = 3$ and $a_4 \ge 5$, then $a_1 \ge 7$ (since $a_1$ must be coprime to $a_3 = 3$) and so
\[
\frac{3}{a_1} + \frac{1}{a_2} + \frac{1}{a_3} \le \frac{3}{7} + \frac{1}{5} + \frac{1}{3} = \frac{101}{105} < 1.
\]
So (\ref{equ:condition}) implies that $a_4 = 4$. Then since $3/8 + 1/4 + 1/3 = 23/24 < 1$, it follows that $a_1 < 8$. Hence if $a_3 = 3$, then $(a_1, a_2, a_3) = (5,4,3)$ or $(7,4,3)$.

Lastly, suppose that $a_3 = 2$. If $a_2 > 5$, then $a_2 \ge 7$ and $a_1 \ge 9$ (since $a_1,a_2$ must be coprime to $a_3$). So
\[
\frac{3}{a_1} + \frac{1}{a_2} + \frac{1}{a_3} \le \frac{3}{9} + \frac{1}{7} + \frac{1}{2} = \frac{41}{42} < 1.
\]
So $a_2 \le 5$, hence $a_2 = 3$ or $5$.

If $a_3 = 2$ and $a_2 = 5$, then (\ref{equ:condition}) implies $a_1 \le 10$, hence $(a_1, a_2, a_3) = (7, 5, 2)$ or $(9, 5, 2)$.

If $a_3 = 2$ and $a_2 = 3$, then (\ref{equ:condition}) implies $a_1 \le 18$, hence $(a_1,a_2,a_3)$ is one of $(5, 3, 2), (7, 3, 2), (11, 3, 2), (13, 3, 2), (17, 3, 2)$.
\end{proof}

\begin{table}
\begin{equation*}
\renewcommand{\arraystretch}{1.4}
\begin{tabular}{|c|c|c|c|c|}
\hline
$(a_1,a_2,a_3)$ & $N$ & ${\rm rk}( HF_{red}^+(-Y) )$ & $p$ & ${\rm rk}( HF_{red}^+(-Y_0) )$ \\
\hline
$(2,3,7)$ & $1$ & $1$ & none & \\
\hline
$(2,3,11)$ & $5$ & $1$ & $5$ & $1$ \\
\hline
$(2,3,13)$ & $7$ & $2$ & $5$ & $0$ \\
&  &  & $7$ & $1$ \\
\hline
$(2,3,17)$ & $11$ & $2$ & $5$ & $1$ \\
&  &  & $7$ & $0$ \\
&  &  & $11$ & $1$ \\
\hline
$(2,5,7)$ & $11$ & $2$ & $3$ & $0$ \\
&  &  & $11$ & $1$ \\
\hline
$(2,5,9)$ & $17$ & $2$ & $7$ & $1$ \\
&  &  & $11$ & $0$ \\
&  &  & $13$ & $0$ \\
&  &  & $17$ & $1$ \\
\hline
$(3,4,5)$ & $13$ & $2$ & $2$ & $0$ \\
&  &  & $7$ & $0$ \\
&  &  & $11$ & $0$ \\
&  &  & $13$ & $1$ \\
\hline
$(3,4,7)$ & $23$ & $2$ & $2$ & $1$ \\
&  &  & $5$ & $0$ \\
&  &  & $11$ & $1$ \\
&  &  & $13$ & $0$ \\
&  &  & $17$ & $0$ \\
&  &  & $19$ & $0$ \\
&  &  & $23$ & $1$ \\
\hline
\end{tabular}
\end{equation*}
\caption{}\label{fig}
\end{table}

\begin{theorem}\label{thm:free2}
We have that ${\rm rk}( HF_{red}^+(-Y) ) > {\rm rk}( HF_{red}^+( -Y_0 , \mathfrak{s}_0 ))$ except in the following cases:
\begin{itemize}
\item[(1)]{$Y = \Sigma(2,3,5)$ and $p$ is any prime.}
\item[(2)]{$Y = \Sigma(2,3,11)$ and $p=5$.}
\end{itemize}
In case (1) we have ${\rm rk}( HF_{red}^+(-Y) ) = {\rm rk}( HF_{red}^+( -Y_0 , \mathfrak{s}_0 )) = 0$ and in case (2) we have ${\rm rk}( HF_{red}^+(-Y) ) = {\rm rk}( HF_{red}^+( -Y_0 , \mathfrak{s}_0 )) = 1$.
\end{theorem}
\begin{proof}
By Lemma \ref{lem:rk}, we only need to consider the case that $r=3$ and up to reordering $(a_1,a_2,a_3)$ is one of:
\[
(2,3,5), (2,3,7), (2,3,11), (2,3,13), (2,3,17), (2,5,7), (2,5,9), (3,4,5), (3,4,7).
\]
In the case $Y = \Sigma(2,3,5)$ we have that ${\rm rk}( HF_{red}^+(-Y) ) = 0$ and $N < 0$. Hence we also have ${\rm rk}( HF_{red}^+(-Y_0) ) = 0$ by Proposition \ref{prop:deltap}.

In the remaining cases we have ${\rm rk}( HF_{red}^+(-Y) ) > 0$ and $N > 0$. Hence if $p > N$ then ${\rm rk}( HF_{red}^+(-Y) ) > {\rm rk}( HF_{red}^+(-Y_0) ) = 0$, by Proposition \ref{prop:deltap}. So we only need to consider primes such that $p \le N$ and coprime to $a_1, a_2, a_3$. This leaves only finitely many cases of $4$-tuples $\{ (a_1,a_2,a_3, p) \}$ to consider. We check each of these cases by directly computing the ranks of $HF_{red}^+(-Y)$ and $HF_{red}^+(-Y_0)$. The results are shown in Table \ref{fig}. By inspection we see that ${\rm rk}( HF_{red}^+(-Y) ) > {\rm rk}( HF_{red}^+( -Y_0 , \mathfrak{s}_0 ))$ in all cases except for $Y = \Sigma(2,3,11)$ and $p=5$.

\end{proof}

\begin{theorem}
Let $Y = \Sigma(a_1 , a_2 , \dots , a_r)$ be a Brieskorn homology sphere and let $p$ be a prime not dividing $a_1 \cdots a_r$. Then $\delta_\infty^{(p)}(Y) > \delta(Y)$ except in the following cases:
\begin{itemize}
\item[(1)]{$Y = \Sigma(2,3,5)$ and $p$ is any prime.}
\item[(2)]{$Y = \Sigma(2,3,11)$ and $p=5$.}
\end{itemize}
In both cases we have $\delta_\infty^{(p)}(Y) = \delta(Y) = 1$.
\end{theorem}
\begin{proof}
Theorems \ref{thm:free} and \ref{thm:free2} imply that $\delta_\infty^{(p)}(Y) \ge \delta(Y)$ with equality only in the cases listed. In these cases, we have $\delta_\infty^{(p)}(Y) = \delta(Y)$ and $\delta(Y) = 1$ \cite[\textsection 8]{os}.
\end{proof}



\bibliographystyle{amsplain}

\end{document}